\documentclass{amsart}
\title{Dense ideals and cardinal arithmetic}
\author{Monroe Eskew}
\usepackage{amssymb}
\usepackage{amsmath} 
\usepackage[utf8]{inputenc} \usepackage{enumerate}
\usepackage{amsthm}
\usepackage{cite}

\newtheorem{theorem}{Theorem}[section]
\newtheorem{lemma}[theorem]{Lemma}
\newtheorem{proposition}[theorem]{Proposition}
\newtheorem{corollary}[theorem]{Corollary}

\newtheorem*{question}{Question}

\newtheorem*{definition}{Definition}

\DeclareMathOperator{\dom}{dom}
\DeclareMathOperator{\ran}{ran}
\DeclareMathOperator{\ot}{ot}

\DeclareMathOperator{\cf}{cf}
\DeclareMathOperator{\cof}{cof}
\DeclareMathOperator{\p}{\mathcal{P}}
\DeclareMathOperator{\den}{d}
\DeclareMathOperator{\sat}{sat}
\DeclareMathOperator{\add}{Add}
\DeclareMathOperator{\col}{Col}
\DeclareMathOperator{\ord}{Ord}
\DeclareMathOperator{\card}{Card}
\DeclareMathOperator{\supp}{supp}
\DeclareMathOperator{\crit}{crit}
\DeclareMathOperator{\id}{id}

\begin{document}
\maketitle

\begin{abstract}From large cardinals we show the consistency of normal, fine, $\kappa$-complete $\lambda$-dense ideals on $\p_\kappa(\lambda)$ for successor $\kappa$.  We explore the interplay between dense ideals, cardinal arithmetic, and squares, answering some open questions of Foreman.
\end{abstract}

Most large cardinals are characterizable in terms of elementary embeddings between models of set theory that have a certain amount of agreement with the full universe $V$.  A typical large cardinal is the least ordinal moved by a nontrivial map $j : V \to M$, where $M$ is a transitive class, and the strength of the large cardinal assumption tends to increase as $M$ gets closer to $V$.  Such cardinals are inaccessible and much more.  This phenomenon can however be realized at small cardinals when the embedding $j : V \to M$ is defined in a forcing extension $V[G]$.  The nature of the forcing adds another dimension to these ``generic large cardinals,'' and their strength tends to increase as the three models $V$, $M$, and $V[G]$ more closely resemble one another.

Here, we consider generic versions of supercompactness at successor cardinals that are optimal in the sense that the forcing poset needed to produce the elementary embedding is of the smallest possible size.  We show that relative to a super-almost-huge cardinal, there can exist a successor cardinal $\kappa$ such that for every regular $\lambda \geq \kappa$, there is a normal, fine, $\kappa$-complete, $\lambda$-dense ideal on $\p_\kappa(\lambda)$.  As far as the author knows, this is the first result exhibiting the consistency of even \emph{saturated} normal and fine ideals on $\p_\kappa(\lambda)$ for a fixed successor $\kappa$ and several values of $\lambda$ simultaneously.   The method used also has immediate application to show the non-absoluteness of some cardinal characteristics of the powerset of a fixed regular cardinal $\mu$, even between models with the same cardinals and same $\mu$-sequences.

Generic large cardinals can have strong influence over the combinatorial structure of the universe in their vicinity.  We explore the interplay between dense ideals, cardinal arithmetic, nonregular ultrafilters, and stationary reflection.  We answer two open questions posed by Foreman in \cite{foremanhandbook} and provide a ``global'' counterexample to an old conjecture in model theory.  We also show some limitations of dense ideals near singular cardinals, establishing the optimality some aspects of our consistency results.  Finally we show that in contrast to traditional supercompactness, the strong forms of generic supercompactness considered here are compatible with Jensen's principle $\square$.

\section{Preliminaries}
\subsection{Forcing}
First we review some essential facts about forcing.  We refer the reader to \cite{jechbook} and \cite{kunen} for background and details.

A partial order $\mathbb{P}$ is said to be \emph{separative} when $p \nleq q \Rightarrow (\exists r \leq p) r \perp q$.  Every partial order $\mathbb{P}$ has a canonically associated equivalence relation $\sim_s$ and a separative quotient $\mathbb{P}_s$, which is isomorphic to $\mathbb{P}$ if $\mathbb{P}$ is already separative.  In most cases we will assume our partial orders are separative.  For every separative partial order $\mathbb{P}$, there is a canonical complete boolean algebra $\mathcal{B}(\mathbb{P})$ with a dense set isomorphic to $\mathbb{P}$.

A map $e: \mathbb{P} \to \mathbb{Q}$ is an \emph{embedding} when it preserves order and incompatibility.  An embedding is said to be \emph{regular} when it preserves the maximality of antichains.  A order-preserving map $\pi : \mathbb{Q} \to \mathbb{P}$ is called a \emph{projection} when $\pi(1_\mathbb{Q}) = 1_\mathbb{P}$, and $p \leq \pi(q) \Rightarrow (\exists q' \leq q) \pi(q') \leq p$.

\begin{lemma}Suppose $\mathbb{P}$ and $\mathbb{Q}$ are partial orders. 

\begin{enumerate}[(1)]
\item $G$ is a generic filter for $\mathbb{P}$ iff $\{ [p]_s : p \in G \}$ is a generic filter for $\mathbb{P}_s$.

\item $e : \mathbb{P} \to \mathbb{Q}$ is a regular embedding iff for all $q \in \mathbb{Q}$, there is $p \in \mathbb{P}$ such that for all $r \leq p$, $e(r)$ is compatible with $q$.

\item The following are equivalent:
\begin{enumerate}[(a)]
\item There is a regular embedding $e : \mathbb{P}_s \to \mathcal{B}(\mathbb{Q}_s)$.
\item There is a projection $\pi : \mathbb{Q}_s \to \mathcal{B}(\mathbb{P}_s)$.
\item There is a $\mathbb{Q}$-name $\dot{g}$ for a $\mathbb{P}$-generic filter such that for all $p \in \mathbb{P}$, there is $q \in \mathbb{Q}$ such that $q \Vdash p \in \dot{g}$.
\end{enumerate}

\item Suppose $e : \mathbb{P} \to \mathbb{Q}$ is a regular embedding.  If $G$ is a filter on $\mathbb{P}$ let $\mathbb{Q}/G = \{ q : \neg \exists p \in G( e(p) \perp q) \}$.  The following are equivalent:
\begin{enumerate}[(a)]
\item $H$ is $\mathbb{Q}$-generic over $V$.
\item $G = e^{-1}[H]$ is $\mathbb{P}$-generic over $V$, and $H$ is $\mathbb{Q} / G$-generic over $V[G]$.
\end{enumerate}

\end{enumerate}
\end{lemma}

\begin{lemma}Suppose $\mathbb{P}$ and $\mathbb{Q}$ are partial orders.  $\mathcal{B}(\mathbb{P}_s) \cong \mathcal{B}(\mathbb{Q}_s)$ iff the following holds.  Letting $\dot{G}, \dot{H}$ be the canonical names for the generic filters for $\mathbb{P},\mathbb{Q}$ respectively, there is a $\mathbb{P}$-name for a function $\dot{f}_0$ and a $\mathbb{Q}$-name for a function $\dot{f}_1$ such that:
\begin{enumerate}[(1)]
\item $\Vdash_\mathbb{P} \dot{f}_0(\dot{G})$ is a $\mathbb{Q}$-generic filter,
\item $\Vdash_\mathbb{Q} \dot{f}_1(\dot{H})$ is a $\mathbb{P}$-generic filter,
\item $\Vdash_\mathbb{P}  \dot{G} = \dot{f}_1^{\dot{f}_0(\dot{G})} (\dot{f}_0(\dot{G}))$, and $\Vdash_\mathbb{Q}  \dot{H} = \dot{f}_0^{\dot{f}_1(\dot{H})} (\dot{f}_1(\dot{H}))$.
\end{enumerate}
An isomorphism is given by $p \mapsto || \check{p} \in \dot{f}_1(\dot{H}) ||_{\mathcal{B}(\mathbb{Q}_s)}$.
\end{lemma}

For a broader notion of ``forcing equivalence,'' the best that can be said in general is the following:

\begin{lemma}Suppose $\mathbb{P}$ and $\mathbb{Q}$ are partial orders.
\begin{enumerate}[(1)]
\item If $e : \mathbb{P} \to \mathbb{Q}$ is a regular embedding, and any $\mathbb{Q}$-generic $H$ yields $V[H] = V[e^{-1}[H]]$, then there is a predense set $A \subseteq \mathcal{B}(\mathbb{Q}_s)$ such that $\mathcal{B}(\mathbb{P}_s) \cong  \mathcal{B}(\mathbb{Q}_s) \restriction a$ for all $a \in A$.

\item $\mathbb{P}$ and $\mathbb{Q}$ yield the same generic extensions iff for a dense set of $p \in \mathbb{P}$, there is $q \in  \mathcal{B}(\mathbb{Q}_s )$ such that  $\mathcal{B}(\mathbb{P}_s) \restriction p  \cong \mathcal{B}(\mathbb{Q}_s )\restriction  q$.
\end{enumerate}
\end{lemma}

A partial order $\mathbb{P}$ is said to be \emph{$\kappa$-distributive} if for any collection of maximal antichains in $\mathbb{P}$, $\{ A_\alpha : \alpha < \beta < \kappa \}$, there is a maximal antichain $A$ such that $A$ refines $A_\alpha$ for all $\alpha < \beta$.  $\mathbb{P}$ is called \emph{$(\kappa,\lambda)$-distributive} if the same holds restricted to antichains of size $\leq \lambda$.  Forcing with $\mathbb{P}$ adds adds no new functions from any $\alpha<\kappa$ to $\lambda$ iff $\mathcal{B}(\mathbb{P})$ is $(\kappa,\lambda)$-distributive.

A strictly stronger property than distributivity is strategic closure.  For a partial order $\mathbb{P}$ and an ordinal $\alpha$, we define a game $G_\alpha(\mathbb{P})$ with two players \emph{Even} and \emph{Odd}.  \emph{Even} starts by playing some element $p_0 \in \mathbb{P}$.  At successor stages $\beta+1$, the next player must play some element $p_{\beta+1} \leq p_\beta$.  \emph{Even} plays at limit stages $\beta$ if possible, by playing a $p_\beta$ that is $\leq p_\gamma$ for all $\gamma < \beta$.  If \emph{Even} cannot play at some stage below $\alpha$, the game is over and \emph{Odd} wins; otherwise \emph{Even} wins.  We say that $\mathbb{P}$ is \emph{$\alpha$-strategically closed} if for every $p \in \mathbb{P}$, \emph{Even} has a winning strategy with first move $p$.  Note that under this definition, every partial order is trivially $\omega$-strategically closed.

A stronger property that $\kappa$-strategic closure is $\kappa$-closure.  $\mathbb{P}$ is \emph{$\kappa$-closed} when any descending chain of length less than $\kappa$ has a lower bound.  $\mathbb{P}$ is \emph{$\kappa$-directed closed} when any directed set of size ${<} \kappa$ has a lower bound.

For any partial order $\mathbb{P}$, the \emph{saturation} of $\mathbb{P}$, $\sat(\mathbb{P})$, is the least cardinal $\kappa$ such that every antichain in $\mathbb{P}$ has size less than $\kappa$.  Erd\H{o}s and Tarski~\cite{ET} proved that $\sat(\mathbb{P})$ is always regular.  The \emph{density} of $\mathbb{P}$, $\den(\mathbb{P})$, is the least cardinality of a dense subset of $\mathbb{P}$.  Clearly $\sat(\mathbb{P}) \leq \den(\mathbb{P})^+$ for any $\mathbb{P}$.  We say $\mathbb{P}$ is \emph{$\kappa$-saturated} if $\sat(\mathbb{P}) \leq \kappa$, and $\mathbb{P}$ is \emph{$\kappa$-dense} if $\den(\mathbb{P}) \leq \kappa$.  A synonym for $\kappa$-saturation is the \emph{$\kappa$-chain condition ($\kappa$-c.c.).}

The properties of distributivity, strategic closure, saturation, and density are robust in the sense that they are absolute between $\mathbb{P}$ and $\mathcal{B}(\mathbb{P})$ for any separative partial order $\mathbb{P}$, and often inherited by intermediate forcings:
\begin{lemma}Suppose $e : \mathbb{P} \to \mathbb{Q}$ is a regular embedding and $\kappa$ is a cardinal.
\begin{enumerate}[(1)]
\item If $\mathbb{Q}$ is $\kappa$-strategically closed, then so is $\mathbb{P}$.
\item $\mathbb{Q}$ is $\kappa$-distributive iff $\mathbb{P}$ is $\kappa$-distributive and $\Vdash_{\mathbb{P}} \mathbb{Q} / \dot{G}$ is $\kappa$-distributive.
\item $\mathbb{Q}$ is $\kappa$-saturated iff $\mathbb{P}$ is $\kappa$-saturated and $\Vdash_{\mathbb{P}} \mathbb{Q} / \dot{G}$ is $\kappa$-saturated.
\item $\mathbb{Q}$ is $\kappa$-dense iff $\mathbb{P}$ is $\kappa$-dense and $\Vdash_{\mathbb{P}} \mathbb{Q} / \dot{G}$ is $\kappa$-dense.
\end{enumerate}
\end{lemma}

For any forcing $\mathbb{P}$ and any $\mathbb{P}$-name $\dot{X}$ for a set of ordinals, there is a canonically associated complete subalgebra $\mathcal{A}_{\dot{X}} \subseteq \mathcal{B}(\mathbb{P})$ that captures $\dot{X}$.  It is the smallest complete subalgebra containing all elements of the form $|| \check{\alpha} \in \dot{X} ||$ for $\alpha$ an ordinal.  $\mathcal{A}_{\dot{X}}$ has the property that whenever $G \subseteq \mathbb{P}$ is generic, $\dot{X}^G$ and $G \cap \mathcal{A}_{\dot{X}}$ are definable from each other using the parameters $\mathcal{B}(\mathbb{P})$ and its powerset, as computed in the ground model.  In this case, we have $V[\dot{X}^G] = V[G \cap \mathcal{A}_{\dot{X}}]$.  See \cite[p. 247]{jechbook} for details.

\subsection{Ideals}

Let $Z$ be any set.  An \emph{ideal} $I$ on $Z$ is a collection of subsets of $Z$ closed under taking subsets and pairwise unions.  If $\kappa$ is a cardinal, $I$ is called \emph{$\kappa$-complete} if it is also closed under unions of size less than $\kappa$.  \emph{``Countably complete''} is taken as synonymous with ``$\omega_1$-complete.''  $I$ is called \emph{nonprincipal} if $\{z \} \in I$ for all $z \in Z$, and \emph{proper} if $Z \notin I$.  Hereafter we will assume all our ideals are nonprincipal and proper.

Let $X = \bigcup Z$.  $I$ is called \emph{fine} if for all $x \in X$, $\{ z : x \notin z \} \in I$.  $I$ is called \emph{normal} if for any sequence $\langle A_x : x \in X \rangle \subseteq I$, the \emph{``diagonal union''} $\{ z : \exists x(x \in z \in A_x) \}$ is in $I$.  It is well-known that $I$ is normal iff for any $A \in \p(Z) \setminus I$ and any function $f$ on $A$ such that $f(z) \in z$ for all $z \in A$, there is an $x$ such that $f^{-1}(x) \notin I$.

To fix notation, let $I^* = \{ Z \setminus A : A \in I \}$ (the \emph{$I$-measure one sets}), $I^+ = \p(Z) \setminus I$ (the \emph{$I$-positive sets}), $\hat{x} = \{ z : x \in z \}$, and denote diagonal unions by $\nabla_{x \in X} A_x$.  Note that $\nabla_{x \in X} A_x = \bigcup_{x \in X} \hat{x} \cap A_x$.

The following basic fact seems to have been previously overlooked--see, for example, the hypotheses of several theorems in \cite{foremanhandbook} and \cite{foremanduality}.

\begin{proposition}All normal and fine ideals are countably complete.
\end{proposition}
\begin{proof}Let $I$ be a normal and fine ideal on $Z \subseteq \p(X)$.  If $\{ x_\alpha : \alpha < \kappa \}$ is an enumeration of distinct elements of $X$, and $\bigcap_{\alpha<\kappa} \hat{x}_\alpha \in I^*$, then $I$ is $\kappa^+$-complete.  For suppose that $\{ A_\alpha : \alpha<\kappa \} \subseteq I$, but $A =  \bigcup_{\alpha<\kappa} A_\alpha \in I^+$.  Then by hypothesis, $A \cap (\bigcap_{\alpha < \kappa} \hat{x}_\alpha) \in I^+$.  Let $f : A \to X$ be defined by $f(z) = x_\alpha$, where $\alpha$ is the least ordinal such that $z \in A_\alpha$.  By normality, there is some $A_\alpha \in I^+$, a contradiction.  So it suffices to find an infinite set $\{x_n : n < \omega \} \subseteq X$ such that $\bigcap_{n<\omega} \hat{x}_n \in I^*$.  Since we assume $I$ is proper and nonprincipal, $X$ is infinite.  We show that any infinite set of distinct elements of $X$ suffices.

Let $\{x_n : n < \omega \}$ be distinct elements of $X$, and suppose the contrary, that $B = \{ z : \{x_n : n < \omega \} \nsubseteq z \} \in I^+$.  By fineness, $B \cap \hat{x}_0 \in I^+$.  For each $z \in B \cap \hat{x}_0$, let $n_z$ be the largest integer such that $\{x_0,...,x_{n_z} \} \subseteq z$.  Let $f : B \cap \hat{x}_0 \to X$ be defined by $f(z) = x_{n_z}$.  By normality, there is an $n$ such that $f^{-1}(x_n) \in I^+$.  Then for all $z \in f^{-1}(x_n)$, $x_{n+1} \notin z$.  This contradicts fineness. 
\end{proof}

\begin{proposition}If $I$ is a normal, fine, $\kappa$-complete ideal on $Z \subseteq \p(\kappa)$, then $\kappa \in I^*$.
\end{proposition}
\begin{proof}Suppose $A = \{ z \in Z : z$ is not an ordinal$\} \in I^+$.  Let $f : A \to \kappa$ be such that $f(z)$ is the least $\alpha \in z$ such that $\alpha \nsubseteq z$.  Then for some $\alpha$, $f^{-1}(\alpha) \in I^+$.  However, $\{ z : \alpha \subseteq z \} \in I^*$ by fineness and $\kappa$-completeness.   
\end{proof}

Proofs of the following facts can be found in \cite{foremanhandbook}.  If $I$ is an ideal on $Z$, say $A \sim_I B$ if the symmetric difference $A \Delta B$ is in $I$.  Let $[A]_I$ denote the equivalence class of $A$ mod $\sim_I$.  The equivalence classes form a boolean algebra under the obvious operations, which we denote by $\p(Z) / I$.  Normality ensures a certain amount of completeness of the algebra:

\begin{proposition}
\label{sums}
Suppose $I$ is a normal and fine ideal on $Z \subseteq \p(X)$.  If $\{ A_x : x \in X \} \subseteq \p(Z)$, then $\nabla A_x$ is the least upper bound of $\{ [A_x]_I : x \in X \}$ in $\p(Z)/I$.
\end{proposition}

If we force with this algebra, we get a generic ultrafilter $G$ on $Z$ extending $I^*$.  We can form the ultrapower $V^Z / G$.  If this ultrapower is well-founded for every generic $G$, then $I$ is called precipitous.  A combinatorial characterization of precipitousness is given by the following:

\begin{theorem}[Jech-Prikry]
\label{prec}
$I$ is a precipitous ideal on $Z$ iff the following holds: For any sequence $\langle A_n : n < \omega \rangle\subseteq \p(I^+)$, such that for each $n$, 
\begin{enumerate}[(1)]
\item $B_n = \{ [a]_I : a \in A_n \}$ is a maximal antichain in $\p(Z) / I$,
\item $B_{n+1}$ refines $B_n$,
\end{enumerate}
there is a function $f$ with domain $\omega$ such that for all $n$, $f(n) \in A_n$, and $\bigcap_{n<\omega} f(n) \not= \emptyset$.
\end{theorem}

For an ideal $I$, the saturation, density, distributivity, and strategic closure of $I$ refers to that of the corresponding boolean algebra.  The next proposition is immediate from Theorem~\ref{prec}:

\begin{proposition}
If $I$ is an $\omega_1$-complete, $\omega_1$-distributive ideal, then $I$ is precipitous.
\end{proposition}

\begin{proposition}Suppose $I$ is a $\kappa$-complete precipitous ideal on $Z$, and there is no $A \in I^+$ such that $I {\restriction} A$ is $\kappa^+$-complete.  Let $G$ be $\p(Z)/I$-generic, and let $j : V \to M$ be the associated elementary embedding, where $M$ is the transitive collapse of $V^Z/G$.  Then the critical point of $j$ is $\kappa$.
\end{proposition}

\begin{proposition}
Let $I$ be an ideal  $Z \subseteq \p(X)$.  Then $I$ is normal and fine iff $1 \Vdash_{\p(Z)/I} [\id]_{\dot G} = j[X]$.
\end{proposition}

\begin{proposition}
Suppose $I$ is an ideal on $Z \subseteq \p(X)$.  If $I$ is $\kappa$-complete and $\kappa^+$-saturated, or if $I$ is normal, fine, and $|X|^+$-saturated, then every antichain in $\p(Z)/I$ has a system of pairwise disjoint representatives. 
\end{proposition}

\begin{proof}
If $I$ is $\kappa$-complete, and $\{ A_\alpha : \alpha < \kappa \}$ is an antichain, replace each $A_\alpha$ with $A_\alpha \setminus \bigcup_{\beta<\alpha} A_\beta$.  If $I$ is normal and fine, and $\{ A_x : x \in X \}$ is an antichain, replace $A_x$ by $A_x \cap \hat{x} \setminus \bigcup_{y \not= x} A_y \cap \hat{y}$.   
\end{proof}

\begin{theorem}
\label{disj}
Suppose $I$ is a countably complete ideal on $Z$, and every antichain in $\p(Z)/I$ has a system of pairwise disjoint representatives.  Then:
\begin{enumerate}[(1)]
\item $I$ is precipitous.
\item $\p(Z)/I$ is a complete boolean algebra.
\item If $G$ is generic over $\p(Z)/I$, $j : V \to M$ is the associated embedding, and $ j[\lambda] \in M$, then $M$ is closed under $\lambda$-sequences from $V[G]$.
\end{enumerate}
\end{theorem}

If $\kappa=\mu^+$ and $I$ is a normal, fine, $\kappa$-complete ideal on $\p_\kappa(\lambda)$, then $I$ is not $\lambda$-saturated.  For otherwise, let $j : V \to M \subseteq V[G]$ be a generic embedding arising from $I$.  Then $M \models |[\id]| = \mu$, and $[\id] =  j[\lambda]$, so $\lambda$ has cardinality $\mu$ in $V[G]$.  So the smallest possible density of such an ideal is $\lambda$.

\subsection{Elementary embeddings}
Proofs of the following can be found in \cite{kanamori}.

\begin{lemma}Suppose $M$ and $N$ are models of $ZF^-$, $j : M \to N$ is an elementary embedding, $\mathbb{P} \in M$ is a partial order, $G$ is $\mathbb{P}$-generic over $M$, and $H$ is $j(\mathbb{P})$-generic over $N$.  Then $j$ has a unique extension $\hat{j} : M[G] \to N[H]$ with $\hat{j}(G) = H$ iff $j[G] \subseteq H$.
\end{lemma}

\begin{lemma}Suppose $M$, $N$ are transitive models of ZFC with the same ordinals, and $j : M \to N$ is an elementary embedding.  Then either $j$ has a critical point, or $j$ is the identity and $M = N$.
\end{lemma}

\section{Dense ideals from large cardinals}

Here we show that it is consistent relative to an almost-huge cardinal that there is a normal, $\kappa$-complete, $\lambda$-dense ideal on $\p_\kappa(\lambda)$, where $\kappa$ is the successor of a regular cardinal $\mu$, and $\lambda \geq \kappa$ is regular, for many particular choices for $\mu,\lambda$.  We also show that relative to a super-almost-huge cardinal, there can exist a successor cardinal $\kappa$ such that for every regular $\lambda \geq \kappa$, there is a normal, $\kappa$-complete, $\lambda$-dense ideal on $\p_\kappa(\lambda)$.  This generalizes a theorem of Woodin about the relative consistency of an $\aleph_1$-dense ideal on $\aleph_1$, and has the following additional advantages: (1) An explicit forcing extension is taken, rather than an inner model of an extension.  (2) Careful constructions within a model where the axiom of choice fails, as presented in~\cite{foremanhandbook}, are avoided.

Let us first recall the essential facts about almost-huge cardinals (see~\cite{kanamori}, 24.11).  A cardinal $\kappa$ is almost-huge if there is an elementary embedding $j : V \to M$ with critical point $\kappa$, such that $M^{<j(\kappa)} \subseteq M$.

\begin{theorem}The following are equivalent:
\begin{enumerate}[(1)]
\item $\kappa$ carries an almost-huge embedding $j$ such that $j(\kappa) = \delta$.
\item $\delta$ is inaccessible, and there is a sequence $\langle U_\alpha : \kappa \leq \alpha < \delta \rangle$ such that:
\begin{enumerate}
\item each $U_\alpha$ is a normal, fine, $\kappa$-complete ultrafilter on $\p_\kappa(\alpha)$,
\item for $\alpha < \beta$, $U_\alpha = \{ A \subseteq \p_\kappa(\alpha) : \{ z \in \p_\kappa(\beta) : z \cap \alpha \in A \} \in U_\beta \}$, and
\item for all $\alpha < \delta$ and all $f : \p_\kappa(\alpha) \to \kappa$ such that $\{ z : f(z) \geq \ot(z) \} \in U_\alpha$, there is $\beta$ such that  $\alpha \leq \beta < \delta$ and $\{ z : f (z \cap \alpha) = \ot(z) \} \in U_\beta$.
\end{enumerate}
\end{enumerate}
Furthermore, if a system as in (2) is given, the direct limit model and embedding witness the almost-hugeness of $\kappa$ with target $\delta$.
\end{theorem}

A system as in (2) will be called an \emph{almost-huge tower}.  Almost-huge towers capture almost-hugeness in a minimal way:

\begin{corollary}
If $\kappa$ has an almost-huge tower of height $\delta$, and $j : V \to M$ is the embedding derived from the tower, then we have $\delta < j(\delta) < \delta^+$, and $j[\delta]$ is cofinal in $j(\delta)$.
\end{corollary}
\begin{proof}
For each $\alpha < \delta$, let $M_\alpha$ be the transitive collapse of $V^{\p_\kappa(\alpha)} / U_\alpha$, and let $j_\alpha : V \to M_\alpha$ and $k_\alpha : M_\alpha \to M$ be the canonical embeddings, with $j = k_\alpha \circ j_\alpha$.  Since $\delta$ is inaccessible, $j_\alpha(\kappa) < \delta$ and $j_\alpha(\delta) = \delta$ for each $\alpha < \delta$.

If $\gamma < j(\delta)$, then there are some $\alpha,\beta < \delta$ such that $k_\alpha(\beta) = \gamma$.  Thus there are only $\delta$ ordinals below $j(\delta)$.  Also, there is $\eta < \delta$ such that $j_\alpha(\eta) > \beta$, so $j(\eta) > \gamma$, and thus $j[\delta]$ is cofinal in $j(\delta)$.  
\end{proof}

A \emph{super-almost-huge} cardinal is a cardinal $\kappa$ such that for all $\lambda \geq \kappa$, there is an almost huge tower of height $\geq \lambda$.  The next result follows from considering the set of closure points under witnesses to property (c) in the tower characterization.

\begin{corollary}
If $\kappa$ has an almost-huge tower of Mahlo height $\delta$,  then for stationary many $\alpha < \delta$, $V_\alpha \models ZFC +  \kappa$ is super-almost-huge.
\end{corollary}

There is a vast gap in strength between almost-huge and huge:

\begin{theorem}
\label{tourney}
If $\kappa$ is a huge cardinal, then there is a stationary set $S \subseteq \kappa$ such that for all $\alpha < \beta$ in $S$, $\alpha$ has an almost-huge tower of height $\beta$.
\end{theorem}

\begin{proof}
Suppose $j : V \to M$ is an elementary embedding with critical point $\kappa$, $j(\kappa) = \delta$, and $M^\delta \subseteq M$.  Then $\kappa$ carries an almost-huge tower $\vec{U}$ of length $\delta$, and $\vec{U} \in M$.  Let $F$ be the ultrafilter on $\kappa$ defined by $F = \{ X \subseteq \kappa: \kappa \in j(X) \}$.  Let $A = \{ \alpha < \kappa : \alpha$ carries an almost-huge tower of height $\kappa \}$.  Since $\kappa \in j(A)$, $A \in F$.  Now let $c: \kappa^2 \to 2$ be defined by $c(\alpha,\beta) = 1$ if $\alpha$ carries an almost-huge tower of height $\beta$, and $c(\alpha,\beta) = 0$ otherwise.  By Rowbottom's theorem, let $H \in F$ be homogeneous for $c$.  We claim $c$ takes constant value 1 on $H$.  For if $\alpha \in A \cap H$, then $\{ \alpha,\kappa \} \in [j(A \cap H)]^2$, and $j(c)(\alpha,\kappa) = 1$.  
\end{proof}

\subsection{Layering and absorption}

\begin{definition}
We will call a partial order $\mathbb{P}$ \emph{$(\mu,\kappa)$-nicely layered} when there is a collection $\mathcal{L}$ of atomless regular suborders of $\mathbb{P}$ such that:
\begin{enumerate}[(1)]
\item for all $\mathbb{Q} \in \mathcal{L}$, $\mathbb{Q}$ is $\mu$-closed and has size $< \kappa$,
\item for all $\mathbb{Q}_0, \mathbb{Q}_1 \in \mathcal{L}$, if $\mathbb{Q}_0 \subseteq \mathbb{Q}_1$, then $\Vdash_{\mathbb{Q}_0} \mathbb{Q}_1 / \dot{G}$ is $\mu$-closed, and
\item for all $\mathbb{P}$-names $\dot{f}$ for a function from $\mu$ to the ordinals, and all $\mathbb{Q}_0 \in \mathcal{L}$, there is an $\mathbb{Q}_1 \in \mathcal{L}$ and an $\mathbb{Q}_1$-name $\dot{g}$ such that $\mathbb{Q}_0 \subseteq \mathbb{Q}_1$, and $\Vdash_\mathbb{P} \dot{f} = \dot{g}$.
\end{enumerate}

We will say $\mathbb{P}$ is \emph{$(\mu,\kappa)$-nicely layered with collapses,} $(\mu,\kappa)$-NLC, when additionally for all $\alpha < \kappa$ and all $\mathbb{Q}_0 \in \mathcal{L}$, there is $\mathbb{Q}_1 \in \mathcal{L}$ such that $\mathbb{Q}_0 \subseteq \mathbb{Q}_1$ and $\Vdash_{\mathbb{Q}_1} |\mathbb{Q}_1| = \mu$.
\end{definition}

\begin{proposition}
If $\mathcal{L}$ witnesses that $\mathbb{P}$ is $(\mu,\kappa)$-nicely layered, then $\mathbb{P}$ is $\kappa$-c.c.\ and $\bigcup \mathcal{L}$ is dense in $\mathbb{P}$.
\end{proposition}

\begin{proof}
Suppose that $\{ p_\alpha : \alpha < \eta \} \subseteq \mathbb{P}$, $\eta \geq \kappa$, is a maximal antichain.  Let $\dot{f}$ be a name of a function with domain $\{ 0 \}$ such that $f(0) = \alpha$ iff $p_\alpha \in G$.  There cannot be a regular suborder $\mathbb{Q}$ of size $< \kappa$ and a $\mathbb{Q}$-name $\dot{g}$ that is forced to be equal to $\dot{f}$, since such a $\dot{g}$ would have $<\kappa$ possible values for its range.

Similarly, let $p \in \mathbb{P}$ be arbitrary, and let $\{ p_\alpha : \alpha < \delta \}$  be a maximal antichain with $p = p_0$.  Let $\dot{f}$ be a name of a function with domain $\{ 0 \}$ such that $f(0) = \alpha$ iff $p_\alpha \in G$. If $\mathbb{Q}$ is a regular suborder and $\dot{g}$ is a $\mathbb{Q}$-name such that $\Vdash_\mathbb{P} \dot{f} = \dot{g}$, then there is some $q \in \mathbb{Q}$ forcing $\dot{g}(0) = 0$, so $q \leq p$.
\end{proof}

\begin{proposition}If there exists a $(\mu,\kappa)$-NLC poset, then $\alpha^{<\mu} < \kappa$ for all $\mu < \kappa$.
\end{proposition}
\begin{proof}
Let $\mathbb P$ be $(\mu,\kappa)$-NLC with layering $\mathcal L$, and let $\alpha < \kappa$.  If $\mathbb Q \in \mathcal L$ collapses $\alpha$ to $\mu$, then we can build a $\mu$-closed tree $T \subseteq \mathbb Q$ of height $\mu$ such that each level is an antichain of size $\geq \alpha$.  $\alpha^{<\mu} \leq |T| < \kappa$. 
\end{proof}

\begin{lemma}[Folklore]
\label{folk}
If $\mathbb{P}$ is a $\mu$-closed partial order such that $\Vdash_\mathbb{P} |\mathbb{P}| = \mu$, then $\mathcal{B}(\mathbb{P}) \cong \mathcal{B}(\col(\mu,|\mathbb{P}|))$.
\end{lemma}

\begin{proof}
Pick a $\mathbb{P}$-name $\dot{f}$ for a bijection from $\mu$ to $\dot{G}$.  We build a tree $T \subseteq \mathbb{P}$ that is isomorphic to a dense subset of $\col(\mu,|\mathbb{P}|)$, and show that it is dense in $\mathbb{P}$.  Each level will be a maximal antichain in $\mathbb{P}$.  Let the first level $T_0 = \{ 1_\mathbb{P} \}$.  If levels $\{ T_\beta : \beta < \alpha + 1 \}$ are defined, below each $p \in T_\alpha$, pick a $|\mathbb{P}|$-sized maximal antichain of conditions deciding $\dot{f}(\alpha)$, and let $T_{\alpha+1}$ be the union of these antichains.  If $\{ T_\beta : \beta < \lambda \}$ is defined up to a limit $\lambda$, pick for each descending chain $b$ through the previous levels, a $|\mathbb{P}|$-sized maximal antichain of lower bounds to $b$, and set $T_\lambda$ equal to the union of these anithchains.  It is easy to check that $T_\lambda$ is a maximal antichain.  Let $T = \bigcup_{\alpha<\mu} T_\alpha$.  To show $T$ is dense, let $p \in \mathbb{P}$.  Let $q \leq p$ be such that for some $\alpha < \mu$, $q \Vdash \dot{f}(\alpha) = p$.   $q$ is compatible with some $r \in T_{\alpha+1}$.  Since $r$ decides $\dot{f}(\alpha)$ and forces it in $\dot{G}$, $r \leq p$.   
\end{proof}

\begin{lemma}
\label{rearrange}
Suppose $\mu < \kappa$ are regular, and $\mathbb{P}$ is $(\mu,\kappa)$-NLC.  If $G$ is $\mathbb{P}$-generic over $V$, then there is a forcing $\mathbb{R} \in V[G]$ such that $\mathbb{R}$ adds a filter $H \subseteq \break \col(\mu,{<}\kappa)$ which is generic over $V$ and such that $(\ord^\mu)^{V[G]} = (\ord^\mu)^{V[H]}$.
\end{lemma}

\begin{proof}
Let $\mathcal{L}$ witness the $(\mu,\kappa)$-NLC property.  In $V[G]$, let $\mathbb{R}$ be the collection of filters $h \subseteq \col(\mu,{<}\alpha)$ for $\alpha< \kappa$ which are generic over $V$, such that for some $\mathbb{Q} \in \mathcal{L}$, $V[h] = V[G \cap \mathbb{Q}]$.  The ordering is end-extension.

Let $h \in \mathbb{R}$ with $\mathbb{Q}_0 \in \mathcal{L}$ a witness, and let and $\alpha < \kappa$ be arbitrary.  Let $\alpha < \beta < \kappa$ and $\mathbb{Q}_1 \supseteq \mathbb{Q}_0$ in $\mathcal{L}$ be such that in $V[h]$, $|\mathbb{Q}_1 / (G \cap \mathbb{Q}_0 ) | = |\beta|$, and $\mathbb{Q}_1$ collapses $\beta$ to $\mu$.  By the definition and Lemma~\ref{folk}, $\mathbb{Q}_1 / (G \cap \mathbb{Q}_0 )$ is equivalent in $V[h]$ to $\col(\mu,\beta)$, which is equivalent to the ${<}\mu$-support product of $\col(\mu,\gamma)$ for $\alpha \leq \gamma \leq \beta$.  The filter $G \cap \mathbb{Q}_1$ therefore gives a filter $h' \supseteq h$ on $\col(\mu,<\beta+1)$ that is generic over $V$, with $V[h'] = V[\mathbb{Q}_1 \cap G]$.  

Let $h \in \mathbb{R}$ with $\mathbb{Q}_0 \in \mathcal{L}$ a witness, and let $f: \mu \to \ord$ in $V[G]$ be arbitrary.  By the definition of $(\mu,\kappa)$-NLC, we can find some $\mathbb{Q}_1 \supseteq \mathbb{Q}_0$ in $\mathcal{L}$ such that $f \in V[G \cap \mathbb{Q}_1]$.  By the previous paragraph, we may find $\mathbb{Q}_2 \supseteq \mathbb{Q}_1$ in $\mathcal{L}$ equivalent to some $\col(\mu,<\alpha)$, and some filter $h' \subseteq \col(\mu,{<}\alpha)$ generic over $V$, extending $h$, and such that $V[G \cap \mathbb{Q}_2] = V[h']$.

If $F$ is generic over $\mathbb{R}$, let $H = \bigcup_{h \in F} h$. Since $\col(\mu,{<}\kappa)$ is $\kappa$-c.c., $H$ is generic, since any maximal antichain from $V$ intersects some $h \in F$.  By the above arguments, any $f : \mu \to \ord$ in $V[G]$ is in $V[H]$.  Conversely, any $f : \mu \to \ord$ in $V[H]$ lives in some $V[h]$ with $h \in \mathbb{R}$, so is in $V[G]$.   
\end{proof}

\subsubsection{The anonymous collapse}

Let $\kappa$ be a regular cardinal whose regularity is preserved by a forcing $\mathbb{P}$.  Let $A(\mathbb{P})$ be the complete subalgebra of $\mathcal{B} ( \mathbb{P} * \add(\kappa))$ generated by the canonical name for the $\add(\kappa)$-generic set.  More precisely, if $e : \mathbb{P}*\add(\kappa) \to  \mathcal{B} ( \mathbb{P} * \add(\kappa))$ is the canonical dense embedding, $A(\mathbb{P})$ is completely generated by the elements of the form $e(\langle 1, \dot{\{ \langle \alpha,1 \rangle \} } \rangle)$.  By \cite[p. 247]{jechbook}, we have a canonical correspondence between such $\add(\kappa)$-generic sets $X$ which come after forcing with $\kappa$-preserving posets $\mathbb P$, and $A(\mathbb P)$-generic filters $H$.  We will move between the two by writing, for example, $X_H$ and $H_X$.

In the case that $\alpha^{<\mu} < \kappa$ for all $\alpha < \kappa$ and $\mathbb{P} = \col(\mu,{<}\kappa)$, denote $A(\mathbb{P})$ by $A(\mu,\kappa)$, and write $B(\mu,\kappa)$ for $\mathcal{B}(\col(\mu,{<}\kappa)*\add(\kappa))$.

\begin{lemma}
\label{quotdist}
If $\mathbb{P}$ is $(\mu,\kappa)$-NLC, and $H \subseteq A(\mathbb{P})$ is generic over $V$, then \break  $\mathcal{B}(\mathbb{P} * \add(\kappa))/H$ is $\kappa$-distributive in $V[H]$.
\end{lemma}

\begin{proof}
$V[H] = V[X_H]$ for the canonically associated $X_H \subseteq \kappa$, and by forcing with $\mathcal{B}(\mathbb{P} * \add(\kappa))/H$ over $V[H]$, we recover a filter $G * X_H$ for $\mathbb{P} * \add(\kappa)$, generic over $V$.
 
If $G * X$ is $\mathbb{P} * \add(\kappa)$-generic over $V$, then $X$ codes all subsets of $\mu$ that live in $V[G]$.  By the definition of $(\mu,\kappa)$-NLC, every $z \in (\ord^\mu)^{V[G]}$ occurs in some submodel of the form $V[G \cap \mathbb{Q}]$, where $\mathbb{Q}$ is isomorphic to $\col(\mu,\alpha)$ for some $\alpha < \kappa$.  Thus $z \in V[y]$ for some $y \subseteq \mu$ in $V[G]$, so $(\ord^\mu)^{V[X]} \supseteq (\ord^\mu)^{V[G]}$.  Since $\add(\kappa)$ adds no $\mu$-sized sets of ordinals, $(\ord^\mu)^{V[G]} = (\ord^\mu)^{V[G*X]}  \supseteq (\ord^\mu)^{V[X]}$.  Thus $\mathcal{B}(\mathbb{P} * \add(\kappa))/H$ is $\kappa$-distributive.  
\end{proof}

\begin{lemma}
\label{cheat}
Let $V$ be a countable transitive model of ZFC (or just assume generic extensions are always available), and assume $\Vdash^V_\mathbb{P} \kappa$ is regular.  If $X \subseteq \kappa$, the following are equivalent:
\begin{enumerate}[(1)]
\item $X$ is $A(\mathbb{P})$-generic over $V$.
\item There is $G \subseteq \mathbb{P}$ such that $G$ is generic over $V$, and $X$ is $\add(\kappa)$-generic over $V(P_0)$, where $P_0 = \p_\kappa(\kappa)^{V[G]}$.
\end{enumerate}
\end{lemma}

\begin{proof}
If $X$ is $A(\mathbb{P})$-generic then force with $\mathcal{B}(\mathbb{P} * \add(\kappa))/H_X$ over $V[X]$, obtaining $G$ such that $G*X$ is $\mathbb{P}*\add(\kappa)$-generic over $V$.  Then $X$ is $\add(\kappa)$-generic over $V[G]$, and since $\add(\kappa)^{V[G]} = \add(\kappa)^{V(P_0)}$, $X$ is $\add(\kappa)$-generic over $V(P_0)$.

Suppose $G \subseteq \mathbb{P}$ is generic over $V$, and $X$ is $\add(\kappa)$-generic over $V(P_0)$, but not $A(\mathbb{P})$-generic over $V$.  Then some $p \in \add(\kappa)^{V(P_0)}$ forces this with $\dom(p) = \alpha < \kappa$, and $X \restriction \alpha = p$.  Take $Y \subseteq \kappa$ such that $Y \restriction \alpha = p$ that is $\add(\kappa)$-generic over the larger model $V[G]$.  Then $Y$ is $A(\mathbb{P})$-generic over $V$, and $V(P_0)[Y]$ can see this, but this contradicts the property of $p$.  So $X$ was $A(\mathbb{P})$-generic over $V$.   
\end{proof}

\begin{theorem}
\label{forget}
For any $\mathbb{P}$ that is is $(\mu,\kappa)$-NLC, there is an isomorphism \break $\iota : A(\mathbb{P}) \to A(\mu,\kappa)$ such that $\iota(|| \alpha \in \dot{X} ||_{A(\mathbb{P})}) = || \alpha \in \dot{X} ||_{A(\mu,\kappa)}$ for all $\alpha < \kappa$.
\end{theorem}

\begin{proof}
Let $X$ be $A(\mathbb{P})$-generic over $V$.  There is a $\kappa$-distributive forcing over $V[X]$ to get $G$ such that $G * X$ is $\mathbb{P} * \add(\kappa)$-generic over $V$.  By Lemma~\ref{rearrange}, we can do further forcing to obtain $H \subseteq \col(\mu,{<}\kappa)$ generic over $V$ such that $(\ord^\mu)^{V[H]} = (\ord^\mu)^{V[G]}$.  By Lemma~\ref{cheat}, $X$ is also $A(\mu,\kappa)$-generic over $V$.

Conversely, every $A(\mu,\kappa)$-generic $X$ is $A(\mathbb{P})$-generic.  For suppose $X$ is a counterexample.  Then there is some $(p,\dot{q}) \in \col(\mu,{<}\kappa) * \add(\kappa)$ such that $(p,\dot{q}) \Vdash \dot{X}$ is not $A(\mathbb{P})$-generic over $V$.  Let $Y$ be any $A(\mathbb{P})$-generic set, and let $P_0 = \p(\mu)^{V[Y]}$.   By the above, $Y$ is $A(\mu,\kappa)$-generic over $V$.  Thus we can force over $V[Y]$ to get $H \subseteq \col(\mu,{<}\kappa)$ such that $H *Y$ is $B(\mu,\kappa)$-generic over $V$.  By the homogeneity of the Levy collapse, there is some automorphism $\pi \in V$ such that $p \in \pi[H] = H^\prime$.   By the homogeneity of Cohen forcing, there is some automorphism $\sigma$ in $V(P_0)$ such that $\sigma[Y]$ is a generic $Y^\prime$ such that $Y^\prime \restriction \dom(\dot{q}^{H^\prime}) = \dot{q}^{H^\prime}$.  $Y^\prime$ is also $A(\mathbb{P})$-generic over $V$.  However, $(p,\dot{q}) \in H^\prime * Y^\prime$, so we have a contradiction.

This implies that we have a canonical correspondence between $A(\mathbb{P})$- and $A(\mu,\kappa)$-generic filters, i.e. definable functions $f,g$ such that for any generic $H$ for $A(\mathbb{P})$, $f(H)$ is the generic for $A(\mu,\kappa)$ computed from $X_H$, and vice versa, and $g(f(H)) = H$.  For $p \in A(\mathbb{P})$, put $\iota(p) = || p \in g(\dot{H}) ||_{A(\mu,\kappa)}$.  It is easy to see that $\iota$ is a complete embedding.  For any $q \in A(\mu,\kappa)$, there is $p \in A(\mathbb{P})$ forces that $q \in f(\dot{H})$.  Thus if $H$ is generic for $A(\mu,\kappa)$ and $\iota(p) \in H$, then $p \in g(H)$, so $q \in f(g(H)) = H$, hence $\iota(p) \leq q$.  The range of $\iota$ is dense, so it is an isomorphism.  By the way we construct $f$ and $g$, $\iota(|| \alpha \in \dot{X} ||_{A(\mathbb{P})}) =  || \alpha \in \dot{X} ||_{A(\mu,\kappa)}$.   
\end{proof}

This machinery has some interesting applications to the absoluteness of some properties of a given powerset.  First, it is easy to see for regular $\mu < \kappa$ such that $\alpha^{<\mu} < \kappa$ for all $\alpha < \kappa$, $\col(\mu,{<}\kappa) \times \add(\mu,\lambda)$ is $(\mu,\kappa)$-NLC for every $\lambda$.  Thus if $X$ is $A(\mu,\kappa)$-generic, then for any $\lambda$, we may further force to obtain a model which is a $(\col(\mu,{<}\kappa) \times \add(\mu,\lambda)) * \add(\kappa)$-generic extension with the same $\ord^\mu$.  Taking inner models given by such $\col(\mu,{<}\kappa) \times \add(\mu,\lambda)$-generic sets, we produce many models with the same cardinals and same $\p(\mu)$, each assigning a different cardinal value for $2^\mu$.  For example, if we add $\omega_1$ Cohen reals to any model of $M$ of ZFC, this is the same as forcing with $\col(\omega,{<}\omega_1)$.  There is for each uncountable ordinal $\alpha \in M$, a generic extension with the same reals and same cardinals, in which it appears we have added $\alpha$ many Cohen reals.

By using weakly compact cardinals, we can get even more dramatic examples.  If $\kappa$ is weakly compact, every $\kappa$-c.c.\ partial order captures small sets in small factors.  To show this, first consider a partial order $\mathbb{P}$ of size $\kappa$.  We can code $\mathbb{P}$ as $A \subseteq \kappa$, and by weak compactness, there is some transitive elementary extension $(V_\kappa,\in,A) \prec (M,\in,B)$.  If $\mu < \kappa$, then any $\mathbb{P}$-name for function $f : \mu \to \ord$ has an equivalent name $\tau \in V_\kappa$ by the $\kappa$-c.c.  Since $A \in M$ and $M$ sees $A$ as a regular suborder of $B$, $M$ thinks that $\tau$ is a $\mathbb{Q}$-name for some regular suborder $\mathbb{Q}$ of $B$.  By elementarity, $V_\kappa$ thinks that $\tau$ is a $\mathbb{Q}$-name for some regular $\mathbb{Q}$ of $A$.  For $\mathbb{P}$ of arbitrary size, let $\tau$ be a $\mathbb{P}$-name of size $<\kappa$, take some regular $\theta$ such that $\mathbb{P},\tau \in H_\theta$, and take an elementary $M \prec H_\theta$ with $\mathbb{P},\tau \in M$ such that $|M| = \kappa$ and $M^{<\kappa} \subseteq M$.  It is easy to see that $M \cap \mathbb{P}$ is a regular suborder of $\mathbb{P}$, and so the above considerations apply to show that there is some regular $\mathbb{Q} \subseteq \mathbb{P} \cap M \subseteq \mathbb{P}$ of size $<\kappa$ such that $\tau$ is a $\mathbb{Q}$-name.

Therefore, if $\kappa$ is weakly compact and $\mathbb{P}$ is $\kappa$-c.c.,\ the collection $\mathcal{L}$ of all regular suborders of $\mathbb{P}$ of size $<\kappa$ witnesses that $\mathbb{P}$ is $(\omega,\kappa)$-nicely layered.  If $\mathbb{P}$ also forces $\kappa = \aleph_1$, then this collection also witnesses that $\mathbb{P}$ is $(\omega,\kappa)$-NLC.  To check this, take any $\mathbb{Q}_0 \in \mathcal{L}$, any $\mathbb{P}$-name $\tau$ of size $<\kappa$, and $\alpha < \kappa$.  Let $H \subseteq \mathbb{Q}_0$ be generic.  Since $\kappa$ is still weakly compact in $V[H]$, there is some regular $\mathbb{Q}_1 \subseteq \mathbb{P}/H$ of size $<\kappa$ in $V[H]$ such that the $(\mathbb{P}/H)$-name associated to $\tau$ is a $\mathbb{Q}_1$-name.  Let $\beta \geq \max \{ \alpha, | \mathbb{Q}_0 * \dot{\mathbb{Q}}_1| \}$.  Since $\mathbb{P} / (\mathbb{Q}_0  * \dot{\mathbb{Q}_1})$ adds a generic for $\col(\omega,\beta)$, we have $\mathbb{Q}_2 \in \mathcal{L}$ extending $\mathbb{Q}_0 * \dot{\mathbb{Q}}_1$ such that $\mathbb{Q}_2 \sim \col(\omega,\beta)$.

In particular, if $\kappa$ is weakly compact, then $\col(\omega,<\kappa) * \dot{\mathbb{Q}}$, where $\dot{\mathbb{Q}}$ is forced to be c.c.c.,\ is $(\omega,\kappa)$-NLC.  Thus an extremely wide variety of forcing extensions with very different theories can be obtained, each sharing the same reals and same cardinals.

\subsubsection{An unfortunate reality}

Despite the universality of $A(\mu,\kappa)$, it is difficult to characterize its combinatorial structure.  While it absorbs all of the small sets added by a $(\mu,\kappa)$-NLC forcing, no such forcing completely embeds into it.  The reader may opt to skip this section, as later results will not depend it.

To show this, we first isolate two properties of a forcing extension that depend on two regular cardinals $\mu <\kappa$.  The author is grateful to Mohammad Golshani for bringing these properties to his attention.
\begin{enumerate}[(1)]
\item \emph{Levy($\mu,\kappa$):} $(\exists A \in [\kappa]^\kappa ) ( \forall y \in [ \kappa ]^\mu \cap V ) y \nsubseteq A$.
\item \emph{Silver($\mu,\kappa$):} $(\exists A \in [\kappa]^\kappa) (\forall X \in [\kappa]^\kappa \cap V)(\exists y \in [X]^\mu \cap V) y \cap A = \emptyset$.
\end{enumerate}

Note that these are both $\Sigma_1$ properties of the parameters $([ \kappa ]^\mu)^V$ and $([ \kappa ]^\kappa)^V$.  For any partial order $\mathbb{P}$, and collection of dense subsets $\mathcal{D} \subseteq \p(\mathbb{P})$ the statement, ``There is a filter $G \subseteq \mathbb{P}$ that is $\mathcal{D}$-generic,'' is also a $\Sigma_1$ property of $\mathbb{P}$ and $\mathcal{D}$.  Now  the following proposition either holds or fails for a given partial order $\mathbb{P}$ and cardinals $\mu < \kappa$:
\[
(*)_{\mu,\kappa} : (\forall X \in [\mathbb{P} ]^\kappa ) ( \exists y \in [X]^\mu ) y \mbox{ has a lower bound in }\mathbb{P}.
\]

\begin{lemma}
If $\mathbb{P}$ is a separative partial order that satisfies $(*)_{\mu,\kappa}$, preserves the regularity of $\kappa$, and such that $\den(\mathbb{P} \restriction p) = \kappa$ for all $p \in \mathbb{P}$, then $\mathbb{P}$ forces $Silver(\mu,\kappa)$.
\end{lemma}

\begin{proof}
Let $\{ p_\alpha : \alpha < \kappa \}$ be a dense subset of $\mathbb{P}$.  Inductively build a dense $D \subseteq \{ p_\alpha : \alpha < \kappa \}$, putting $p_\alpha \in D$ just in case there is no $\beta < \alpha$ such that $p_\beta \in D$ and $p_\beta \leq p_\alpha$.  $D$ has the property that for all $p \in D$, $| \{ q \in D : p \leq q \} | < \kappa$.  Fixing a bijection $f : D \to \kappa$, we claim that if $G  \subseteq \mathbb{P}$ is generic, $A = f[G]$ witnesses $Silver(\mu,\kappa)$.  Note that since $\mathbb{P}$ is nowhere $<\kappa$-dense, $A$ is an unbounded subset of $\kappa$.  Now let $p \in D$ and $X = \{ q_\alpha : \alpha < \kappa \} \in [D]^\kappa$ be arbitrary.  There is some $B \in [\kappa]^\kappa$ such that for all $\alpha \in B$, $p \nleq q_\alpha$.  For each $\alpha \in B$, choose $r_\alpha \leq p$ such that $r_\alpha \perp q_\alpha$.  By $(*)$, there is some $y \in [B]^\mu$ such that $\{ r_\alpha : \alpha \in y \}$ has a lower bound $r$.  We have $r \Vdash \{ q_\alpha : \alpha \in \check{y} \} \cap \dot{G} = \emptyset$.  As $p$ and $X$ were arbitrary, $Silver(\mu,\kappa)$ is forced.
\end{proof}

\begin{lemma}
If $\mathbb{P}$ is a $\kappa$-c.c.\ separative partial order of size $\kappa$ satisfying $\neg (*)_{\mu,\kappa}$, then some $p \in \mathbb{P}$ forces $Levy(\mu,\kappa)$.
\end{lemma}

\begin{proof}
Suppose $X \in [\mathbb{P}]^\kappa$ witnesses $\neg (*)_{\mu,\kappa}$.  By the $\kappa$-c.c., there is some $p$ such that $p \Vdash | \check{X} \cap \dot{G} | = \kappa$.  If $y \in [X]^\mu$, then $1 \Vdash \check{y} \nsubseteq \dot{G}$, since otherwise some $q$ is a lower bound to $y$.  Hence $p$ forces that $X \cap G$ witnesses $Levy(\mu,\kappa)$.   
\end{proof}

\begin{lemma}Suppose $\mu < \kappa$, $\mu$ is regular for all $\alpha < \kappa$, $\alpha^\mu < \kappa$.  There are two $(\mu,\kappa)$-NLC partial orders $\mathbb{P}_0$ and $\mathbb{P}_1$ such that $\mathbb{P}_0$ forces $Levy \wedge \neg Silver$, and $\mathbb{P}_1$ forces $\neg Levy \wedge Silver$.
\end{lemma}

\begin{proof}Let $\mathbb{P}_0$ be the Levy collapse $\col(\mu, {<}\kappa)$, and let $\mathbb{P}_1$ be the Silver collapse,
\[ \{ p : (\exists \alpha < \mu)(\exists x \in [\kappa]^\mu) p : x \times \alpha \to \kappa, \text{ and } p(\beta,\gamma) < \beta \text{ for all } (\beta,\gamma) \in \dom p \} \]
It is easy to see that $\mathbb{P}_1$ satisfies $(*)_{\mu,\kappa}$, while $\mathbb{P}_0$ fails this property, as witnessed by $X = \mathbb{P}_0$.  Hence by the previous lemmas, $\mathbb{P}_0$ forces $Levy(\mu,\kappa)$, and $\mathbb{P}_1$ forces $Silver(\mu,\kappa)$.  We must show that the respective negations are also forced.

Let $\dot{A}$ be a $\mathbb{P}_0$-name such that $1 \Vdash \dot{A} \in [\kappa]^\kappa$.  Let $p \in \mathbb{P}$ be arbitrary, and let $\gamma< \kappa$ be such that $\supp(p) \subseteq \gamma$.  Let $X_0 = \{ \alpha < \kappa : p \nVdash \alpha \notin \dot{A} \}$.  For each $\alpha \in X_0$, pick some $q_\alpha \leq p$ such that $q_\alpha \Vdash \alpha \in \dot{A}$.  By a delta-system argument, let $X_1 \in [X_0]^\kappa$ be such that there is $r \leq p$ such that for all $\alpha \in X_1$,  $q_\alpha \restriction \gamma = r$, and for $\alpha \not= \beta$ in $X_1$, $(\supp(q_\alpha) \setminus \gamma) \cap (\supp(q_\beta) \setminus \gamma) = \emptyset$.  For any $q \leq r$ and $y \in [X_1]^\mu$, $q \nVdash \check{y} \cap \dot{A} = \emptyset$.  This is because for such $q$, there is some $\alpha \in y$ such that $(\supp(q_\alpha) \setminus \gamma) \cap \supp(q) = \emptyset$, so $q$ is compatible with $q_\alpha$.  Hence $r \Vdash (\exists X \in [\kappa]^\kappa \cap V)(\forall y \in [X]^\mu \cap V) y \cap \dot{A} \not= \emptyset$.  As $\dot{A}$ and $p$ were arbitrary, $\neg Silver(\mu,\kappa)$ is forced.

Now let $\dot{A}$ be a $\mathbb{P}_1$-name such that $1 \Vdash \dot{A} \in [\kappa]^\kappa$, and let $p \in \mathbb{P}_1$ be arbitrary.  Form $X_0$, $\{q_\alpha : \alpha \in X_0 \}$, and $X_1$ like above.  We can take a $y \in [X_1]^\mu$ such that $\bigcup_{\alpha \in y} q_\alpha = q \in \mathbb{P}_1$.  Then $q \Vdash \check{y} \subseteq \dot{A}$, so $q$ forces $\neg Levy(\mu,\kappa)$.   
\end{proof}

\begin{corollary}Suppose $\mu$, $\kappa$, $\mathbb{P}_0$, and $\mathbb{P}_1$ are as above.  Let $G$ be $\mathbb{P}_0$-generic and $H$ be $\mathbb{P}_1$-generic over $V$.  Let $\mathbb{Q} \in V$ be a partial order.  If $\mathbb{Q}$ forces $Levy(\mu,\kappa)$, then $V[H]$ has no $\mathbb{Q}$-generic, and if $\mathbb{Q}$ forces $Silver(\mu,\kappa)$, then $V[G]$ has no $\mathbb{Q}$-generic.  If $\mathbb{Q}$ is $\kappa$-c.c.\ and of size $\kappa$, then no $\kappa$-closed forcing extension of $V[G]$ or $V[H]$ can introduce a generic for $\mathbb{Q}$.
\end{corollary}

\begin{proof}
Since $V[H]$ satisfies $\neg Levy$, and $Levy$ is a $\Sigma_1$ property with parameters in $V$, no inner model of $V[H]$ containing $V$ can satisfy $Levy$.  Likewise, no inner model of $V[G]$ containing $V$ can satisfy $Silver$.  To see that the non-existence of $\mathbb{Q}$-generics is preserved by $\kappa$-closed forcing, suppose that for some such forcing $\mathbb{R} \in V[G]$, $r \Vdash^{V[G]}_\mathbb{R} \dot{K}$ is $\mathbb{Q}$-generic over $V$.  Since $\mathbb{Q}$ has size $\kappa$, we can build a descending sequence $\{ r_\alpha : \alpha < \kappa \}$ below $r$ such that for all $q \in \mathbb{Q}$, there is $r_\alpha$ deciding whether $q \in K$.  Let $K' = \{ q : (\exists \alpha < \kappa) r_\alpha \Vdash q \in \dot{K} \}$.  Any maximal antichain $A \in V$ contained in $\mathbb{Q}$ has size $< \kappa$, thus some $r_\alpha$ completely decides $A \cap K$.  Since $r_\alpha \Vdash \check{A} \cap \dot{K} \not= \emptyset$, we must have $K' \cap A \not= \emptyset$, so $K'$ is $\mathbb{Q}$-generic over $V$.  The argument for $\kappa$-closed forcing over $V[H]$ is the same.   
\end{proof}

\begin{theorem}Suppose $\mu < \kappa$ are regular and $\alpha^\mu<\kappa$ for all $\alpha<\kappa$.  No $(\mu,\kappa)$-NLC forcing regularly embeds into $A(\mu,\kappa)$.  Further, a generic extension by $A(\mu,\kappa)$ has no generic filters for any $\kappa$-c.c.\ forcing $\mathbb{Q}$ such that $\den(\mathbb{Q} \restriction q) \geq \kappa$ for all $q \in \mathbb{Q}$.
\end{theorem}
\begin{proof}
First note that we only need to consider $\mathbb{Q}$ such that $\den(\mathbb{Q} \restriction q) = \kappa$ for all $q \in \mathbb{Q}$.  For if $p \in A(\mu,\kappa)$ is such that $p \Vdash \dot{K}$ is $\mathbb{Q}$-generic, then there would be some $q \in \mathcal{B}(\mathbb{Q})$ and some $p' \leq p$ such that $\mathcal{B}(\mathbb{Q}) \restriction q$ completely embeds into $A(\mu,\kappa) \restriction p'$.  Since $d(A(\mu,\kappa) = \kappa$, this implies $\mathcal{B}(\mathbb{Q}) \restriction q \leq \kappa$.

Let $\mathbb{Q}$ be any $\kappa$-c.c.\ forcing such that $\den(\mathbb{Q} \restriction q) = \kappa$ for all $q \in \mathbb{Q}$.  For any $p \in \mathbb{Q}$, if $(*)$ holds for $\mathbb{Q} \restriction p$, then $p \Vdash Silver$, and otherwise for some $q \leq p$, $q \Vdash Levy$.  Thus $\Vdash_\mathbb{Q} Levy \vee Silver$.  Suppose $K$ is $\mathbb{Q}$-generic over $V$, and $X$ is $A(\mu,\kappa)$-generic over $V$.  There are two further forcings $\mathbb{R}_0,\mathbb{R}_1$ over $V[X]$ that respectively get filters $G,H$ such that $V[G][X]$ is $\mathbb{P}_0 * \add(\kappa)$-generic, and $V[H][X]$ is $\mathbb{P}_1 * \add(\kappa)$-generic.  If $V[K] \models Levy$, then $K \not \in V[H][X]$, and if $V[K] \models Silver$, then $K \notin V[G][X]$.  Thus $V[X]$ has no $\mathbb{Q}$-generics.
\end{proof}

\subsection{Construction of a dense ideal}

First we will define a useful strengthening of ``nicely layered.''

\begin{definition}
$\mathbb{P}$ is \emph{$(\mu,\kappa)$-very nicely layered (with collapses)} when there is a sequence $\langle \mathbb{Q}_\alpha : \alpha < \kappa \rangle = \mathcal{L}$ such that:
\begin{enumerate}[(1)]
\item $\mathcal{L}$ witnesses that $\mathbb{P}$ is $(\mu,\kappa)$-nicely layered (with collapses),
\item $\mathcal{L}$ is $\subseteq$-increasing,
\item every subset of $\mathbb{P}$ of size $<\mu$ with a lower bound has an infimum, and
\item there is a system of continuous projection maps $\pi_\alpha : \mathbb{P} \to \mathbb{Q}_\alpha$ such that for each $\alpha$, $\pi_\alpha \restriction \mathbb{Q}_\alpha = \id$, and for $\beta < \alpha < \kappa$, $\pi_\beta = \pi_\beta \circ \pi_\alpha$.

(By continuous, we mean that for any $X \subseteq \mathbb P$, if $\inf(X)$ exists, then for all $\alpha<\kappa$, $\inf(\pi_\alpha[X]) = \pi_\alpha(\inf(X))$.)
\end{enumerate}
\end{definition}

A typical example is the Levy collapse $\col(\mu,{<}\kappa)$.  In the general case, we will usually abbreviate the action of the projection maps $\pi_\alpha(q)$ by $q {\restriction} \alpha$.  In applying clause (3), we will use the next proposition, proof of which is left to the reader.

\begin{proposition}
If $\mathbb{P}$ is a partial order such that every descending chain of length $< \mu$ has an infimum, then every directed subset of size $<\mu$ has an infimum.
\end{proposition}

\begin{theorem}
\label{maindense}
Assume $\kappa$ carries an almost-huge tower of height $\delta$, and let $j : V \to M$ be given by the tower.  Let $\mu,\lambda$ be regular such that $\mu < \kappa \leq \lambda < \delta$.  Suppose $\Vdash_{A(\mu,\kappa)}$ `` $\dot{ \mathbb{P} }$ is $(\kappa,\delta)$-very nicely layered and forces $\delta = \lambda^+$.''  If $X * H$ is $A(\mu,\kappa) * \dot{\mathbb{P}}$-generic, then in $V[X][H]$, there is a normal, $\kappa$-complete, $\lambda$-dense ideal on $\p_\kappa(\lambda)$.
\end{theorem}

\begin{proof}
Let $H_X$ be the $A(\mu,\kappa)$-generic filter computed from $X$. Let $K \times C$ be $B(\mu,\kappa)/H_X \times \col(\mu,\lambda)$-generic over $V[X][H]$, and for brevity let $W = V[X][H][K][C]$.  Note that $V[X][K] = V[G][X]$, where $G * X$ is some \break $\col(\mu,{<}\kappa) * \add(\kappa)$-generic filter over $V$.  Let $\langle \mathbb{Q}_\alpha : \alpha < \delta \rangle$ witness that $\mathbb{P}$ is $(\kappa,\delta)$-nicely layered.  By the distributivity of $B(\mu,\kappa)/H_X$ in $V[X]$, $\mathbb{P}$ and its layers $\mathbb{Q}_\alpha$ are still $\kappa$-closed in $V[G][X]$.  For $\alpha < \beta$, the relation $\Vdash_{\mathbb{Q}_\alpha}$ ``$\mathbb{Q}_\beta / \mathbb{Q}_\alpha$ is $\kappa$-closed'' holds in $V[G][X]$ because in $V[X]$, $B(\mu,\kappa)/H_X \times \mathbb{Q}_\alpha$ is $\kappa$-distributive.  Furthermore, since no sequences of length $<\mu$ are added, the forcing given by the definition of $\col(\mu,\lambda)$ is the same between $V$, $W$, and intermediate models.

The forcing to get from $V[G]$ to $W$ is equivalent to $(\add(\kappa) \times \col(\mu,\lambda)) *  \mathbb{P}$.  Let $\mathcal{L}$ be the collection of subforcings of the form $(\add(\kappa) \times \col(\mu,\lambda)) * \mathbb{Q}_\alpha$ for $\alpha < \delta$.  This sequence then witnesses the $(\mu,\delta)$-NLC property in $V[G]$.  The closure properties are evident, and since the whole forcing has the $\delta$-c.c., functions from $\mu$ to ordinals are indeed captured by these factors.

Let $P_0 = \p(\mu)^W$, and consider the submodel $M(P_0)$.  In $W$, $Q_0 = \break \p(\add(\delta))^{M(P_0)}$ has cardinality $\delta$.  To show this, let $Y \subseteq \delta$ be $\add(\delta)$-generic over $W$.  By Theorem~\ref{forget}, $Y$ is $A(\mu,\delta)$-generic over $V$, and hence over $M$ since $(\col(\mu,{<}\delta)*\add(\delta))^M = (\col(\mu,{<}\delta)*\add(\delta))^V$ by the closure of $M$.  Since $M[Y]$ thinks $j(\delta)$ is inaccessible, $M[Y] \models |Q_0| < j(\delta)$, so $W[Y] \models |Q_0| = \delta$ since $j(\delta) < (\delta^+)^V$.  Since $W \models 2^\mu = \delta$, $W$ and $W[Y]$ have the same cardinals, so $W \models |Q_0| = \delta$.   Therefore, working in $W$, we can inductively build a set $\hat{X} \subseteq \delta$ that is $\add(\delta)$-generic over $M(P_0)$ with $\hat{X} \cap \kappa = X$.  By Lemma~\ref{cheat}, $\hat{X}$ is $A(\mu,\delta)$-generic over $M[G]$.  A further forcing produces $G^\prime \supseteq G$, such that $G^\prime * \hat{X}$ is $\col(\mu,{<}\delta)*\add(\delta)$-generic over $M$, so we have an elementary $\hat{j} : V[G][X] \to M[G^\prime][\hat{X}]$ extending $j$.  By elementarity, for the corresponding filters $H_X$ and $H_{\hat{X}}$ on the respective algebras $A(\mu,\kappa)^V$ and $A(\mu,\delta)^M$, we have $j[H_X] \subseteq H_{\hat{X}}$.  Hence we can define in $W$ the restricted elementary embedding $\hat{j} : V[X] \to M[\hat{X}]$.

Now we wish to extend $\hat{j}$ to have domain $V[X][H]$.  As in the argument for Lemma~\ref{quotdist}, every element of $(\ord^\mu)^W$ is coded by some element of $M$ and some $y \subseteq \mu$ coded in $\hat{X}$, so $M[\hat{X}]$ is closed under ${<}\delta$-sequences from $W$.  Consequently, $H \cap \mathbb{Q}_\alpha$ and $\hat{j}[ H \cap \mathbb{Q}_\alpha]$ are in $M[\hat{X}]$ for all $\alpha < \delta$.  Also, $M[\hat{X}] \vDash$ ``$\hat{j}(\mathbb{P})$ is $(\delta,j(\delta))$-very nicely layered.''  Each $\hat{j}[ H \cap \mathbb{Q}_\alpha ]$ is a directed set of size $\mu$ in $M[\hat{X}]$, so it has an infimum $m_\alpha \in \hat{j}(\mathbb{Q}_\alpha)$.

Let $\langle A_\alpha : \alpha < \delta \rangle \in W$ enumerate the maximal antichains of  $\hat{j}(\mathbb{P})$ from $M[\hat{X}]$.  (There are only $\delta$ many because $M[\hat{X}]$ thinks this partial order has inaccessible size $j(\delta)$ and is $j(\delta)$-c.c.)  Inductively define an increasing sequence of ordinals $\langle \alpha_i \rangle_{i < \delta} \subseteq \delta$, and a corresponding decreasing sequence of conditions $\langle p_i \rangle_{i < \delta} \subseteq \hat{j}(\mathbb{P})$ as follows.

Assume as the induction hypothesis that we have defined the sequences up to $i$, and for all $\xi < i$ and all $\alpha < \delta$, $p_\xi$ is compatible with $m_\alpha$, and for all $\xi < i$, there is some $a \in A_\xi$ such that $p_\xi \leq a$.  Let $q_i = \inf_{\xi<i} p_\xi$.  This is compatible with all $m_\alpha$ because for all $\alpha$, $\langle p_\xi \wedge m_\alpha : \xi < i \rangle$ is a descending chain in $\hat{j}(\mathbb{P})$.  Let $\alpha_i \geq \sup_{\xi<i} \alpha_\xi$ be such that $A_i \subseteq \hat{j}(\mathbb{Q}_{\alpha_i})$ and $q_i \in \hat{j}(\mathbb{Q}_{\alpha_i})$.  This is possible by the chain condition and because $j[\delta]$ is cofinal in $j(\delta)$.  Choose $p_i \in \hat{j}(\mathbb{Q}_{\alpha_i})$ below $q_i \wedge m_{\alpha_i}$ and some $a \in A_i$.  $p_i$ is compatible with all $m_\alpha$, because for any $\alpha > \alpha_i$, $m_\alpha \restriction j(\alpha_i) = m_{\alpha_i}$.  This is because for any $\alpha < \beta < \delta$, 
\begin{align*}
m_\beta \restriction j(\alpha) = & (\inf \{ j(p) : p \in H {\restriction} \beta \}) \restriction j(\alpha) = \inf \{ j(p) {\restriction} j(\alpha) : p \in H {\restriction} \beta \} \\
= & \inf \{ j(p {\restriction} \alpha) : p \in H {\restriction} \beta \} = \inf \{ j(p) : p \in H {\restriction} \alpha \} = m_\alpha. 
\end{align*}

The upward closure of the sequence $\langle p_i \rangle_{i < \delta}$ is a filter $\hat{H}$ which is $\hat{j}(\mathbb{P})$-generic over $M[\hat{X}]$.  For all $p \in H$, $\hat{j}(p) \in \hat{H}$ since there is some $m_\alpha \leq \hat{j}(p)$.  Thus we get an extended elementary embedding $\hat{j} : V[X][H] \to M[\hat{X}][\hat{H}]$.  In $W$, we define an ultrafilter $U$ over $(\p(\p_\kappa \lambda ))^{V[X][H]}$: let $A \in U$ iff $j[\lambda] \in \hat{j}(A)$.  Note that $j[\lambda] \in \p_{j(\kappa)}(j(\lambda))^{M[\hat{X}][\hat{H}]}$.  $U$ is $\kappa$-complete and normal with respect to functions in $V[X][H]$.  If $f : \p_\kappa(\lambda) \to \lambda$ is a regressive function in $V[X][H]$ on a set $A \in U$, then $\hat{j}(f)(j[\lambda]) = j(\alpha)$ for some $\alpha < \lambda$, so $\{ z \in A : f(z) = \alpha \} \in U$.

Now the forcing to obtain $U$ was $\mathbb{Q} = B(\mu,\kappa)/H_X \times \col(\mu,\lambda)$, the product of a $\kappa$-dense and a $\lambda$-dense partial order.  In $V[X][H]$, let $e : \p(\p_\kappa \lambda) \to \mathcal{B}(\mathbb{Q})$ be defined by $e(A) = || \check{A} \in \dot{U} ||$.  Let $I$ be the kernel of $e$.  $I$ is clearly a normal, $\kappa$-complete ideal.  $e$ lifts to a boolean embedding of $\p(\p_\kappa \lambda) / I$ into $\mathcal{B}(\mathbb{Q})$.  Since $\mathbb{Q}$ is $\lambda^+$-c.c., $I$ is $\lambda^+$-saturated.  If $\langle [A_\alpha] : \alpha < \lambda \rangle$ is a maximal antichain in $\p_\kappa(\lambda) / I$, then $\nabla A_\alpha$ is the least upper bound and is in the dual filter to $I$. $e(\nabla A_\alpha) = || \nabla A_\alpha \in \dot{U} || = 1$, and this is the least upper bound in $\mathcal{B}(\mathbb{Q})$ to $\{ e(A_\alpha) : \alpha <\lambda \}$.  This is because if there were a generic extension in which all $A_\alpha \notin U$, then $\nabla A_\alpha \notin U$ as well since $U$ is normal with respect to sequences from $V[X][H]$.  Therefore $e$ is a complete embedding, and thus $I$ is $\lambda$-dense.  
\end{proof}

We can also characterize the exact structure of $\p(\p_\kappa \lambda) / I$.  First note the following about the ground model embedding $j : V \to M$.  $M$ is the direct limit of the coherent system of $\alpha$-supercompactness embeddings $j_\alpha : V \to M_\alpha$ for $\alpha < \delta$.  Every member of $M_\alpha$ is represented as $j_\alpha(f)(j_\alpha[\alpha])$ for some function $f \in V$ with domain $\p_\kappa(\alpha)$.  If $k_\alpha : M_\alpha \to M$ is the factor map such that $j = k_\alpha \circ j_\alpha$, then the critical point of $k_\alpha$ is above $\alpha$, so $k_\alpha(x) = k_\alpha[x]$ when $M_\alpha \vDash |x| \leq |\alpha|$.  Since $M$ is the direct limit, for any $x \in M$, there is some $\alpha < \delta$ and some $f \in V$ such that
\[ x = k_\alpha([f]) = k_\alpha(j_\alpha(f)(j_\alpha[\alpha])) = j(f)(k_\alpha(j_\alpha[\alpha])) = j(f)(j[\alpha])).
\]

Let $U \subseteq \p(\p_\kappa \lambda) / I$ be generic over $V[X][H]$, and let $j_U : V \to N$ be the generic ultrapower embedding.  Since $e :  \p(\p_\kappa \lambda) / I \to \mathcal{B}(\mathbb{Q})$ is a complete embedding, forcing with $\mathcal{B}(\mathbb{Q}) / e[U]$ over $V[X][H][U]$ produces a model $W$ as above.  Notice that the definition of $e$ and $U$ makes $A \in U$ iff $j[\lambda] \in \hat{j}(A)$.  Hence we can define an elementary embedding $k : N \to M[\hat{X}][\hat{H}]$ by $k([f]) = \hat{j}(f)(j[\lambda])$, and we have $\hat{j} = k \circ j_U$.

What is the critical point of $k$?  Since $N \vDash \mu^+ = \delta$, certainly it must be at least $\delta$.  Let $\beta$ be any ordinal.  There is some $\alpha$ such that $\lambda \leq \alpha < \delta$ and some $f \in V$ such that $\beta = j(f)(j[\alpha])$.  Let $b : \lambda \to \alpha$ be a bijection in $V[X][H]$.  Then $\beta = j(f)(\hat{j}(b)[j[\lambda]])$.  Furthermore, $j[\lambda] = k(j_U[\lambda])$.  Therefore, $\beta = k( j_U(f)(j_U(b)[j_U[\lambda]]))$.  Thus $\beta \in \ran(k)$, and so $k$ does not have a critical point.

Therefore, $N = M[\hat{X}][\hat{H}]$.  By the closure of $M[\hat{X}][\hat{H}]$, the generic $K \times C$ for $\mathbb{Q}$ is in $M[\hat{X}][\hat{H}] = N \subseteq V[X][H][U]$.  So the quotient $\mathcal{B}(\mathbb{Q}) / e[U]$ is trivial and $\p(\p_\kappa \lambda) / I \cong \mathcal{B}(\mathbb{Q}) \restriction q$ for some $q$.   

The generic embeddings coming from $I$ extend the original almost-hugeness embedding.  In particular, $j[\delta]$ is cofinal in $j(\delta)$.  This can also be deduced from the assumption that there is some $A \in I^*$ of size $\lambda$, which of course follows from $\lambda^{<\kappa} = \lambda$.  In contrast, Burke and Matsubara~\cite{BM} proved that if there is a normal, fine, $\kappa$-complete, $\lambda^+$-saturated ideal on $\p_\kappa(\lambda)$ and $\cf(\lambda)<\kappa$, then it is forced that $\sup (j[\lambda^+]) < j(\lambda^+)$.  It seems to be unknown whether it is consistent to have saturated ideals on $\p_\kappa(\lambda)$ for successor $\kappa$ and singular $\lambda$, and this result suggests that quite different methods will be needed for an answer.

\subsubsection{Minimal generic supercompactness}
Generalizing supercompactness, we will say cardinal $\kappa$ is \emph{generically supercompact} when for every $\lambda \geq \kappa$, there is a forcing $\mathbb{P}$ such that whenever $G \subseteq \mathbb{P}$ is generic, there is an elementary embedding $j : V \to M$, where $M$ is a transitive class in $V[G]$, $\crit(j) = \kappa$, $j(\kappa) > \lambda$, and $M^\lambda \cap V[G] \subseteq M$.  We note that unlike in the case of non-generic supercompactness, the condition that $j[\lambda] \in M$ does not imply that $M$ is closed under $\lambda$-sequences from $V[G]$.  Whenever a supercompact $\kappa$ is turned into a successor cardinal by a $\kappa$-c.c.\ forcing, we'll have that for all $\lambda \geq \kappa$, there is a normal, fine, precipitous ideal on $\p_\kappa(\lambda)$ whose generic embeddings always extend the original supercompactness embedding.  But if $j : V \to M$ is an embedding coming from a normal ultrafilter on $\p_\kappa(\lambda)$, then $2^{\lambda^{<\kappa}} < j(\kappa) < (2^{\lambda^{<\kappa}})^+$.  If $\kappa = \mu^+$ in a generic extension $V[H]$, and a further extension gives $\hat{j} : V[H] \to M[\hat{H}] \subseteq V[H][G]$ extending $j$, then $M[\hat{H}]$ is not closed under $\lambda$-sequences from $V[H][G]$.  This is because $| \lambda | = | j(\kappa)| = \mu$ in $V[H][G]$, while $M[\hat{H}]$ thinks $j(\kappa)$ is a cardinal.

Stronger properties of ideals on $\p_\kappa(\lambda)$ are needed to give genuine generic supercompactness.  One such property is $\lambda^+$-saturation, which is implied by $\lambda$-density.  We now sketch how to get a model in which there is a successor cardinal $\kappa$ such that for all regular $\lambda \geq \kappa$, there is a normal, $\kappa$-complete, $\lambda$-dense ideal on $\p_\kappa(\lambda)$.  Start with a super-almost-huge cardinal $\kappa$ and a regular $\mu < \kappa$.  The first part of the forcing is $A(\mu,\kappa)$.  Then we do a proper class iteration, which we prefer to describe instead as an iteration up to an inaccessible $\delta > \kappa$ such that $V_\delta \vDash \kappa$ is super-almost-huge.

Let $T = \{ \alpha < \delta : \kappa$ carries an almost-huge tower of height $\alpha \}$.  Let $C$ be the closure of $T$, and let $\langle \alpha_\beta \rangle_{\beta < \delta}$ be its continuous increasing enumeration.  Over $V^{A(\mu,\kappa)}$, let $\mathbb{P}_\delta$ be the Easton-support limit of the following:
\begin{itemize}
\item Let $\mathbb{P}_0 = \col(\kappa,<\alpha_0)$.
\item If $\beta$ is zero or a successor ordinal, let $\mathbb{P}_{\beta+1} = \mathbb{P}_{\beta} * \col(\alpha_\beta,<\alpha_{\beta+1})$.
\item If $\beta$ is a limit ordinal such that $\alpha_\beta$ is singular, let $\mathbb{P}_{\beta+1} = \col(\alpha_\beta^+,<\alpha_{\beta+1})$.
\item If $\beta$ is a limit ordinal such that $\alpha_\beta$ is regular, let $\mathbb{P}_{\beta+1} = \col(\alpha_\beta,<\alpha_{\beta+1})$.
\end{itemize}

It is routine to verify that this iteration preserves the regularity of the members of $T$, the successors of the singular limit points of $T$, and the regular limit points of $T$.  Further, the set of non-limit-points of $T$ becomes the set of successors of regular cardinals between $\kappa$ and $\delta$.

Let $X \subseteq \kappa$ be $A(\mu,\kappa)$-generic over $V$, and let $H \subseteq \mathbb{P}_\delta$ be generic over $V[X]$. Suppose $\kappa \leq \lambda < \delta$, and $\lambda$ is regular in $V[X][H]$.  Then there is some successor ordinal $\beta<\delta$ such that $\alpha_\beta \in T$ and $\alpha_\beta = \lambda^+$.  Consider the subforcing $A(\mu,\kappa) * \mathbb{P}_\beta = (A(\mu,\kappa) * \mathbb{P}_{\beta-1}) * \col(\lambda,<\alpha_\beta)$.  The forcing $\mathbb{P}_\beta$ is $(\kappa,\alpha_\beta)$-very nicely layered in $V[X]$.

If $j : V \to M_\beta$ is an almost-huge embedding with critical point $\kappa$ and $j(\kappa) = \alpha_\beta$, then by Theorem~\ref{maindense}, there is a normal, $\kappa$-complete, $\lambda$-dense ideal on $\p_\kappa(\lambda)$ in $V[X][H_\beta]$.  Now note that the tail-end forcing $\mathbb{P}_{\beta,\delta}$ is $\alpha_\beta$-closed.  Since $\lambda^{<\kappa} = \lambda$ in $V[X][H_\beta]$, no new subsets of $\p_\kappa(\lambda)$ are added by the tail.  The collection $\{ A_\alpha : \alpha < \lambda \}$ witnessing the $\lambda$-density of $I$ retains this property, as this is a local property of the boolean algebra $\p_\kappa(\lambda) / I$ and $\{ A_\alpha : \alpha < \lambda \}$.  Normality and completeness of $I$ are likewise preserved.  

This method is quite flexible, and can done by iterating collapsing posets other than the Levy collapse, or by using products rather than iterations.

\subsubsection{Dense ideals on successive cardinals?}

At the time of this writing, it is unknown whether there can exist simultaneously a normal $\kappa$-dense ideal on $\kappa$ and a normal $\kappa^+$-dense ideal on $\kappa^+$.  The following is the current best approximation.

Suppose $\langle \kappa_n : n < \omega \rangle$ is a sequence of cardinals such that for all $n$, $\kappa_n$ carries an almost-huge tower of height $\kappa_{n+1}$.  Such a sequence will be called an \emph{almost-huge chain}.  Obviously, extending this to sequences of length longer than $\omega$ requires an extra idea; perhaps we just stack one $\omega$-chain above another, or maybe postulate some relationship between the $\omega$-chains.  By Theorem~\ref{tourney}, such chains occur quite often below a huge cardinal.

Suppose $\langle \kappa_n : 0 < n < \omega \rangle$ is an almost-huge chain, and $\mu < \kappa_1$ is regular.  Consider the full-support iteration $\mathbb{P}$ of $\langle \mathbb{P}_n : n < \omega \rangle$, where $\mathbb{P}_0 = A(\mu,\kappa_1)$, and for all $n < \omega$, $\mathbb{P}_{n+1} = \mathbb{P}_n * A(\kappa_n,\kappa_{n+1})$.  The stage $\mathbb{P}_1 = A(\mu,\kappa_1) * A(\kappa_1,\kappa_2)$ regularly embeds into $A(\mu,\kappa_1) * (\col(\kappa_1,< \! \kappa_2) * \add(\kappa_2))$.  The first two stages here add a normal $\kappa_1$-dense ideal on $\kappa_1$ and make $\kappa_1=\mu^+$, $\kappa_2 = \mu^{++}$.  The third stage preserves this since it adds no subsets of $\kappa_1$.  By Lemma~\ref{quotdist}, the quotient forcing $\mathbb{Q}$ to get from $V^{\mathbb{P}_1}$ to this three-stage extension is $\kappa_2$-distributive.  Now the tail-end forcing $\mathbb{P} / \mathbb{P}_1$ is $\kappa_2$-strategically closed.  Since $\mathbb{Q}$ does not add any plays of the relevant game of length $< \! \kappa_2$, $\mathbb{P} / \mathbb{P}_1$ remains $\kappa_2$-strategically closed in $V^{\mathbb{P}_1 * \mathbb{Q}}$, so forcing with it preserves the $\kappa_1$-dense ideal on $\kappa_1$.   Also, $\mathbb{Q}$ remains $\kappa_2$-distributive in $V^{\mathbb{P}}$, since $\mathbb{Q} \times (\mathbb{P} / \mathbb{P}_1)$ is $\kappa_2$-distributive in $V^{\mathbb{P}_1}$.  It thus remains the case in $V^\mathbb{P}$ that there is a $\kappa_2$-distributive forcing adding a normal $\kappa_1$-dense ideal on $\kappa_1$.

Similarly, consider $V^{\mathbb{P}_n}$ for $n >1$.  $\mathbb{P}_n = \mathbb{P}_{n-2} * (A(\kappa_{n-1}, \kappa_n) * A(\kappa_n,\kappa_{n+1}))$.  Since $|\mathbb{P}_{n-2}| = \kappa_{n-1}$ (or $\mu$ for $n=2$), $\kappa_{n}$ retains an almost-huge tower of height $\kappa_{n+1}$ in $V^{\mathbb{P}_{n-2}}$.  Thus the same argument applies: In $V^{\mathbb{P}_n}$, there is a $\kappa_{n+1}$-distributive forcing adding a normal $\kappa_n$-dense ideal on $\kappa_n$, and this remains true in $V^\mathbb{P}$.  Therefore, we obtain a model in which for all $n >0$, there is a $\mu^{+n+1}$-distributive forcing adding a normal $\mu^{+n}$-dense ideal on $\mu^{+n}$.  By repeating this with a tall enough stack of almost-huge chains, we obtain the consistency of ZFC with the statement, ``For all regular cardinals $\kappa$, there is a $\kappa^{++}$-distributive forcing adding a normal $\kappa^+$-dense ideal on $\kappa^+$.''

\section{Structural constraints}

Saturated ideals have a strong influence over the combinatorial structure of the universe in their vicinity.  Phenomena of this type may also be viewed as the universe imposing constraints on the structural properties of ideals.  Below are some of the most interesting known results to this effect.  Proofs can be found in~\cite{foremanhandbook}.

\begin{enumerate}[(1)]
\item (Tarski) If $I$ is a nowhere-prime ideal which is $\kappa$-complete and $\mu$-saturated for some $\mu < \kappa$, then $2^{<\mu} \geq \kappa$.
\item (Jech-Prikry) If $\kappa = \mu^+$, $2^\mu = \kappa$, and there is a $\kappa$-complete, $\kappa^+$-saturated ideal on $\kappa$, then $2^\kappa = \kappa^+$.
\item (Jech-Prikry) If $\kappa = \mu^+$, and there is a $\kappa$-complete, $\kappa^+$-saturated ideal on $\kappa$, then there are no $\kappa$-Kurepa trees.
\item (Woodin) If there is a countably complete, $\omega_1$-dense ideal on $\omega_1$, then there is a Suslin tree.
\item (Woodin) If there is a countably complete, uniform, $\omega_1$-dense ideal on $\omega_2$, then $2^\omega = \omega_1$.  (Uniform means that all sets of size $<\omega_2$ are in the ideal--equivalent to fineness.)
\item (Shelah) If $2^\omega < 2^{\omega_1}$, then $NS_{\omega_1}$ is not $\omega_1$-dense.
\item (Gitik-Shelah) If $I$ is a $\kappa$-complete, nowhere-prime ideal, then $\den(I) \geq \kappa$.
\end{enumerate}

We note that result (2) easily generalizes to the following: If $\kappa = \mu^+$, $2^\mu = \kappa$, and there is a normal, fine, $\kappa$-complete, $\lambda^+$-saturated ideal on $\p_\kappa(\lambda)$, then $2^\lambda = \lambda^+$.

If no requirements are made for the ideal $I$ and the set $Z$ on which it lives, almost no structural constraints on quotient algebras remain.  The following strengthens a folklore result, probably known to Sikorski.  The argument was supplied by Don Monk in personal correspondence.

\begin{proposition}
Let $\mathbb{B}$ be a complete boolean algebra, and let $\kappa$ be a cardinal such that $2^\kappa \geq |\mathbb{B}|$.  There is a uniform ideal $I$ on $\kappa$ such that $\mathcal{B} \cong \p(\kappa)/I$.
\end{proposition}

\begin{proof}
Let $\kappa$, $\mathbb{B}$ be as hypothesized.  By the theorem of Fichtenholz-Kantorovich and Hausdorff (see \cite[Lemma 7.7]{jechbook}), there exists a family $F$ of $2^\kappa$ many subsets of $\kappa$ such that for any $x_1,...,x_n,y_1,...,y_m \in F$, $x_1 \cap ... \cap x_n \cap (\kappa \setminus y_1) \cap ... \cap (\kappa \setminus y_m)$ has size $\kappa$.  $F$ generates a free algebra: closing $F$ under finitary set operations gives a family of sets $G$ such that any equation holding between elements of $G$ expressed as boolean combinations of elements of $F$ holds in all boolean algebras.  If we pick any surjection $h_0 : F \to \mathbb{B}$ and extend it to $h_1 : G \to \mathbb{B}$ in the obvious way, then $h_1$ will be a well-defined homomorphism.

Let $I_{bd}$ be the ideal of bounded subsets of $\kappa$.  Since all elements of $G$ are either empty or have cardinality $\kappa$, $G \cong G/I_{bd}$, so $h_1$ has an extension $h_2$ from the algebra generated by $G \cup I_{bd}$ to $\mathbb{B}$, where $h_2(x) = 0$ for all $x \in I_{bd}$.  Finally, by Sikorski's extension theorem, there is a further extension to a homomorphism $h_3 : \p(\kappa) \to \mathbb{B}$.  The kernel of $h_3$ is an ideal $I$ such that $\p(\kappa)/I \cong \mathbb{B}$.   
\end{proof}

\subsection{Cardinal arithmetic and ideal structure}

A careful examination of the proof of Woodin's theorem (5) shows that $\omega_2$ can be replaced by any $\omega_n$, $2 \leq n < \omega$.  Aside from that, Woodin's argument is rather specific to the cardinals involved.  In \cite{foremanhandbook}, Foreman asked (Open Question 27) whether the analogous statement holds one level up:

\begin{question}[Foreman]Does the existence of an $\omega_2$-complete, $\omega_2$-dense, uniform ideal on $\omega_3$ imply that $2^{\omega_1} = \omega_2$?
\end{question}

To answer this, we invoke an easy preservation lemma about ideals under small forcing.  If $I$ is an ideal, $\mathbb{P}$ is a partial order, and $G \subseteq \mathbb{P}$ is generic, then $\bar{I}$ denotes the ideal generated by $I$ in $V[G]$, i.e. $\{ X : (\exists Y \in I) X \subseteq Y \}$.

\begin{lemma}
\label{pres}
Suppose $I$ is a $\kappa$-complete ideal on $Z \subseteq \p(X)$, $\mathbb{P}$ is a partial order, and $G$ is $\mathbb{P}$-generic.
\begin{enumerate}[(1)]
\item If $\sat(\mathbb{P}) \leq \kappa$, then $\bar{I}$ is $\kappa$-complete in $V[G]$.
\item If $\den(\mathbb{P}) < \kappa$, then $\den(\bar{I})^{V[G]} \leq \den(I)^V$.  
\end{enumerate}
\end{lemma}

\begin{proof}
For (1), let $\dot{s}$ be a $\mathbb{P}$-name for a sequence of elements of $\bar{I}$ of length less than $\kappa$.  By $\kappa$-saturation, let $\beta < \kappa$ be such that $1 \Vdash dom(\dot{s}) \leq \beta$.  For each $\alpha < \beta$, let $A_\alpha$ be a maximal antichain such that for $p \in A_\alpha$, $p \Vdash \dot{s}(\alpha) \subseteq \check{b}^p_\alpha$, where ${b}^p_\alpha \in I$.  Then $B = \bigcup_{p, \alpha} b^p_{\alpha} \in I$, and $1 \Vdash \bigcup \dot{s} \subseteq \check{B}$.

For (2), let $D \subseteq \mathbb{P}$ be a dense set of size less than $\kappa$, and let $A \in \bar{I}^+$.  Then $A = \bigcup_{d \in D \cap G} \{ z : d \Vdash z \in \dot{A} \}$.  By (1), there is some $d \in D$ such that $\{ z : d \Vdash z \in \dot{A} \} \notin I$.  This shows that $(\p(Z)/I)^V$ is dense in $(\p(Z)/\bar{I})^{V[G]}$, and the conclusion follows.  
\end{proof} 

\begin{corollary}
\label{wg}
If there is a $\kappa^+$-complete, $\kappa^+$-dense, uniform ideal on $\kappa^{++}$, then $2^\kappa = \kappa^+$.
\end{corollary}

\begin{proof}
Suppose for a contradiction that $f: \p(\kappa) \to \kappa^{++}$ is a surjection.  Let $\mathbb{P} = \col(\omega,\kappa)$, and let $G$ be $\mathbb{P}$-generic.  Since $\den(\mathbb{P}) = \kappa$, Lemma~\ref{pres} implies that $\bar{I}$ is $\kappa^+$-complete and $\kappa^+$-dense in $V[G]$.  Furthermore, in $V[G]$, $\kappa^+ = \omega_1$ and $\kappa^{++}= \omega_2$.  Thus Woodin's theorem implies that $V[G] \vDash$ CH.  However, $f$ witnesses the failure of CH, a contradiction.  
\end{proof}

Another interesting constraint can be derived from the following:

\begin{theorem}[Shelah~\cite{shelahproper}]
\label{shelahcof}
Suppose $V \subseteq W$ are models of ZFC.  If $\kappa$ is a regular cardinal in $V$, and $\cf(\kappa) \not= \cf(|\kappa|)$ in $W$, then $(\kappa^+)^V$ is not a cardinal in $W$.
\end{theorem}

\begin{corollary}[Burke-Matsubara \cite{bmclub}]
\label{cofcon}
If $\kappa = \mu^+$, $\lambda \geq \kappa$ is regular, and $I$ is a normal, fine, $\kappa$-complete, $\lambda^+$-saturated ideal on $\p_\kappa(\lambda)$, then $\{ z : \cf(z) = \cf(\mu) \} \in I^*$.
\end{corollary}
\begin{proof}
Let $G$ be a generic ultrafilter extending $I^*$.  Since $\crit(j) = \kappa$ and $\lambda^+$ is preserved, $j(\kappa) = \lambda^+$, and $|\lambda| = \mu$ in $V[G]$.  By Shelah's theorem, $\cf(\lambda) = \cf(\mu)$ in $V[G]$ and in the ultrapower $M$ since $M^\mu \cap V[G] \subseteq M$.  Since $1 \Vdash [\id] = j[\lambda]$, \L o\'{s}'s theorem gives $\{ z : \cf(z) = \cf(\mu) \} \in I^*$.   
\end{proof}

\begin{theorem}
\label{dist}
Suppose $\kappa = \mu^+$, and $I$ is a normal, fine, $\kappa^+$-saturated ideal on $\kappa$.  Then $\p(\kappa) / I$ is $\cf(\mu)$-distributive iff $\mu^{<\cf(\mu)} = \mu$.
\end{theorem}

\begin{proof}
Suppose $\p(\kappa) / I$ is $\cf(\mu)$-distributive, and let $\{ f_\alpha : \alpha < \delta \}$ be an enumeration of $[\mu]^{<\cf(\mu)}$, where $\delta$ is a cardinal.  If $\mu < \delta$, then for any $\p(\kappa) / I$-generic $G$, $([\mu]^{<\cf(\mu)})^V$ is a proper subset of $([\mu]^{<\cf(\mu)})^{V[G]}$, since $j[\delta] \not= j(\delta)$.  This contradicts the distributivity of  $\p(\kappa) / I$.

Since $\p(\kappa)/I$ is $\kappa^+$-saturated, it is $\cf(\mu)$-distributive iff it is $(\cf(\mu), \kappa)$- \break distributive.  Let $G$ be $\p(\kappa) / I$-generic and let $M$ be the generic ultrapower.  Let $\beta < \cf(\mu)$, and suppose $f \in V[G]$ is a function from $\beta$ to $\kappa$.   By Theorem~\ref{disj}, $f \in M$.  By Corollary~\ref{cofcon}, $M \vDash \cf(\kappa) = \cf([\id]) = \cf(\mu)$.  Thus there is a $\gamma < \kappa$ such that $\ran(f) \subseteq \gamma$.  Observe that $j(^\beta \gamma) = (^\beta \gamma)^M = (^\beta \gamma)^V$, since $\mu^\beta < \kappa$.  Hence $f \in V$.  
\end{proof}

\subsection{Stationary reflection}
A stationary subset $S$ of a regular cardinal $\kappa$ is said to reflect if there is some $\alpha < \kappa$ of uncountable cofinality such that $S \cap \alpha$ is stationary in $\alpha$.  A collection of stationary subsets $\{ S_i : i < \delta \}$ of $\kappa$ is said to reflect simultaneously if there is some $\alpha < \kappa$ if $S_i \cap \alpha$ is stationary for all $i < \delta$.  It is well known that if $\kappa = \mu^+$ and $X$ is a set of regular cardinals below $\mu$, then the statement that every stationary subset of $\{ \alpha < \kappa : \cf(\alpha) \in X \}$ reflects contradicts $\square_\mu$, and the statement that every pair of stationary subsets of $\{ \alpha < \kappa : \cf(\alpha) \in X \}$ reflect simultaneously contradicts the weaker principle $\square(\kappa)$.

\begin{theorem}
Suppose there is a $\kappa^{+}$-complete, $\kappa^{++}$-saturated, uniform ideal on $\kappa^{+n}$ for some $n \geq 2$.  Then for $2 \leq m \leq n$, every collection of $\kappa$ many stationary subsets of $\kappa^{+m}$ contained in $\cof(\leq \kappa)$ reflects simultaneously. 
\end{theorem}

\begin{proof}
Suppose $I$ is such an ideal and $j : V \to M \subseteq V[G]$ is a generic embedding arising from the ideal.  The critical point of $j$ is $\kappa^+$, and all cardinals above $\kappa^{+}$ are preserved.  Since $I$ is uniform, and there is a family of $\kappa^{+n+1}$ many almost-disjoint functions from $\kappa^{+n}$ to $\kappa^{+n}$, $j(\kappa^{+n}) \geq (\kappa^{+n+1})^V$.  The first $n-1$ cardinals in $V$ above $\kappa$ must map onto the first $n-1$ cardinals in $M$ above $\kappa$.  But in $M$, there are at least $n-1$ cardinals in the interval $(\kappa,(\kappa^{+n+1})^V)$ since all cardinals above $\kappa^+$ are preserved.  Thus if $j(\kappa^{+n}) > (\kappa^{+n+1})^V$, then $\kappa^{+n+1}$ would be collapsed.  So for $1 \leq m \leq n$, $j(\kappa^{+m}) = (\kappa^{+m+1})^V$.

Let $\{ S_\alpha : \alpha < \kappa \}$ be stationary subsets of $\kappa^{+m}$ concentrating on $\cof(\leq \kappa)$, where $2 \leq m \leq n$.  By the $\kappa^{++}$-chain condition, these sets remain stationary in $V[G]$.  By the above remarks, $\gamma = \sup (j[\kappa^{+m}]) < j(\kappa^{+m})$.  For each $\alpha$, $j \restriction S_\alpha$ is continuous since $\kappa < \crit(j)$.  For each $\alpha$, let $C_\alpha$ be the closure of $S_\alpha$.  In $V[G]$, we can define a continuous increasing function $f : C_\alpha \to \gamma$ extending $j \restriction S_\alpha$ by sending $\sup (S_\alpha \cap \beta)$ to $\sup (j[S_\alpha \cap \beta])$ when $\beta$ is a limit point of $S_\alpha$. This shows that $j[S_\alpha]$ is stationary in $\gamma$.  Now $M$ may not have $j[S_\alpha]$ as an element, but it satisfies that $j(S_\alpha) \cap \gamma$ is stationary in $\gamma$.  Furthermore, $j( \{ S_\alpha : \alpha < \kappa \}) = \{ j(S_\alpha) : \alpha < \kappa \}$, and $M$ sees that these all reflect at $\gamma$.  By elementarity, the $S_\alpha$ have a common reflection point.  
\end{proof}

\begin{proposition}
\label{Z-refl}
Suppose $\mu,\kappa,\lambda$ are regular cardinals such that $\omega < \mu < \kappa = \mu^+ < \lambda$, and $I$ is an ideal on $\p_\kappa(\lambda)$ as in Theorem~\ref{maindense}.  Then every collection $\{ S_i : i < \mu \}$ of stationary subsets of $\lambda \cap \cof( \omega)$ reflects simultaneously.
\end{proposition}

\begin{proof}The algebra $\p(\p_\kappa(\lambda)) / I$ is isomorphic to $\mathcal{B}(\mathbb{P} \times \mathbb{Q})$, where $\mathbb{P}$ is $\kappa$-dense and $\Vdash_\mathbb{P} ``\mathbb{Q}$ is $\mu$-closed.''  Forcing with $\mathbb{P} \times \mathbb{Q}$ thus preserves the stationarity of any subset of $\lambda \cap \cof(\omega)$.  If $j : V \to M \subseteq V[G]$ is a generic embedding arising from the ideal, then since $j[\lambda] \in M$ and $M$ thinks $j(\lambda)$ is regular, $\gamma = \sup(j[\lambda])< j(\lambda)$.  The restriction of $j$ to each $S_i$ is continuous, and as above we may define in $V[G]$ a continuous increasing function from the closure of $S_i$ into $\gamma$, showing $j[S_i]$ is stationary in $\gamma$ for each $i$.  Thus $M \models (\forall i < \mu) j(S_i) \cap \gamma$ is stationary, so by elementarity, the collection reflects simultaneously.   
\end{proof}

\subsection{Nonregular ultrafilters}
The computation of the cardinality of ultrapowers is an old problem of model theory.  Originally, it was conjectured that if $\mu,\kappa$ are infinite cardinals, and $U$ is a countably incomplete uniform ultrafilter on $\kappa$, then $| \mu^\kappa / U | = \mu^\kappa$ \cite{ck}.  It was shown by Donder~\cite{donder} that this conjecture holds in the core model below a measurable cardinal.  A key tool in computing the size of ultrapowers is the notion of regularity:

\begin{definition}
An ultrafilter $U$ on $Z$ is called \emph{$(\mu, \kappa)$-regular} if there is a sequence $\langle A_\alpha : \alpha < \kappa \rangle \subseteq U$ such that for any $Y \subseteq \kappa$ of order type $\mu$, $\bigcap_{\alpha \in Y} A_\alpha = \emptyset$.
\end{definition}

\begin{theorem}[Keisler \cite{keisler}]
\label{keisler}
Suppose $U$ is a $(\mu, \kappa)$-regular ultrafilter on $Z$, witnessed by $\langle A_\alpha : \alpha < \kappa \rangle$.  For each $z \in Z$, let $\beta_z = \ot( \{ \alpha : z \in A_\alpha \}) < \mu$.  Then for any sequence of ordinals $\langle \gamma_z : z \in Z \rangle$, we have $| \prod \gamma_z^{\beta_z} / U | \geq | \prod \gamma_z / U |^\kappa$.
\end{theorem}

Obviously any uniform ultrafilter on a cardinal $\kappa$ is $(\kappa,\kappa)$-regular.  Also, any fine ultrafilter on $\p_\kappa(\lambda)$ is $(\kappa,\lambda)$-regular, as witnessed by $\langle \hat{\alpha} : \alpha < \lambda \rangle$.  Much can be seen by exploiting a connection between dense ideals and nonregular ultrafilters.

\begin{lemma}[Huberich \cite{huberich}]
\label{huberich}
Suppose $\mathbb{B}$ is a complete boolean algebra of density $\kappa$, where $\kappa$ is regular.  Then there is an ultrafilter $U$ on $\mathbb{B}$ such that whenever $X \subseteq \mathbb{B}$ and $\sum X \in U$, then there is $Y \subseteq X$ such that $|Y| < \kappa$ and $\sum Y \in U$.
\end{lemma}

\begin{proof}
Let $D = \{ d_\alpha : \alpha < \kappa \}$ be dense in $\mathbb{B}$.  For any maximal antichain $A \subseteq \mathbb{B}$, let $\gamma_A > 0$ be least such that for all $\alpha < \gamma_A$, there are $\beta < \gamma_A$ and $a \in A$ such that $d_\beta \leq d_\alpha \wedge a$.  Let $C_A = \{ d \in D \restriction \gamma_A : (\exists a \in A) d \leq a \}$.  Let $F = \{ \sum C_A : A$ is a maximal antichain$\}$.

We claim $F$ has the finite intersection property.  Let $A_1,...,A_n$ be maximal antichains.  We may assume $\gamma_{A_1} \leq ... \leq \gamma_{A_n}$.  Let $d_{\alpha_1} \leq d_0 \wedge a_1$ for some $a_1 \in A_1$, where $\alpha_1 < \gamma_{A_1}$.  Let $d_{ \alpha_2} \leq d_{\alpha_1} \wedge a_2$ for some $a_2 \in A_2$, where $\alpha_2 < \gamma_{A_2}$.  Proceeding inductively, we get a descending chain $d_{\alpha_1} \geq ... \geq d_{\alpha_n}$, where each $d_{\alpha_i} \in C_{A_i}$.  Thus $d_{\alpha_n} \leq \sum C_{A_1} \wedge ... \wedge \sum C_{A_n}$.

Let $U \supseteq F$ be any ultrafilter.  If $\sum X \in U$, then we can find an antichain $A$ that is maximal below $\sum X$ such that $(\forall a \in A)(\exists x \in X) a \leq x$.  Extending $A$ it to a maximal antichain $A'$, we have $\sum C_{A'} \in F$.  Since $|C_{A'}| < \kappa$, the conclusion follows.   
\end{proof}

If $I$ is an ideal on $Z$ and $U'$ is an ultrafilter on $\p(Z)/I$, then $U'$ generates an ultrafilter $U \supseteq I^*$ on $Z$ by taking $U = \{ X : [X]_I \in U' \}$.

\begin{lemma}
\label{singreg}
Suppose $\kappa=\mu^+$, $\lambda$ is regular, and $I$ is a normal and fine, $\kappa$-complete, $\lambda$-dense ideal on $Z \subseteq \p_\kappa(\lambda)$.  Then any ultrafilter $U \supseteq I^*$ given by Lemma \ref{huberich} is $(\cf(\mu)+1,\lambda)$-regular.
\end{lemma}

\begin{proof}
Let $U'$ be an ultrafilter on $\p(Z)/I$ given by Lemma \ref{huberich} and let $U$ be the corresponding ultrafilter on $Z$.  By Corollary~\ref{cofcon}, $\{ z : \cf(z) = \cf(\mu) \} \in I^*$.  For such $z$, choose $A_z \subseteq z$ of order type $\cf(\mu)$ that is cofinal in $z$.    We will inductively build a sequence of intervals $\{ (x_\alpha,y_\alpha) : \alpha < \lambda \}$, each contained in $\lambda$, such that $y_\alpha < x_\beta$ when $\alpha < \beta$, and such that for all $\alpha$, $\{ z : A_z \cap (x_\alpha,y_\alpha) \not= \emptyset \} \in U$.

Suppose we have constructed the intervals up to $\beta$.  Let $\lambda > x_\beta > \sup \{ y_\alpha : \alpha < \beta \}$.  For $z \in \hat{x}_\beta$, let $y_\beta(z) \in z$ be such that $A_z \cap (x_\beta,y_\beta(z)) \not= \emptyset$.  Since $I$ is normal, there is a maximal antichain $A$ of $I$-positive sets such that for all $a \in A$, $y_z(\beta)$ is the same for all $z \in a$.  There is some $A' \subseteq A$ of size $< \lambda$ such that $\sum A' \in U'$.  Let $y_\beta > x_\beta$ be such that for $z \in a \in A'$, $y_\beta(z) < y_\beta$.

For $\alpha < \lambda$, let $X_\alpha = \{ z : A_z \cap (x_\alpha,y_\alpha) \not= \emptyset \}$.  Since each $A_z$ has order type $\cf(\mu)$ and the intervals $(x_\alpha,y_\alpha)$ are disjoint and increasing, each $A_z$ cannot have nonempty intersection with all intervals in some sequence of length greater than $\cf(\mu)$.  Thus if $s \subseteq \lambda$ and $z \in \bigcap_{\alpha \in s} X_\alpha$, then $\ot(s) \leq \cf(\mu)$. 
\end{proof}

\begin{lemma}
\label{ultsize}
Suppose $I$ is a $\lambda$-dense, $\kappa$-complete ideal on $Z$, and $\p(Z)/I$ is a complete boolean algebra.  Then any $U \supseteq I^*$ given by Lemma \ref{huberich} has the property that for all $\alpha < \kappa$, $| \alpha^Z / U | \leq 2^{<\lambda}$.
\end{lemma}

\begin{proof}
Let $U'$ be an ultrafilter on $\p(Z)/I$ given by Lemma \ref{huberich} and let $U$ be the corresponding ultrafilter on $Z$.
To compute a bound on $|\alpha^Z/U|$ for $\alpha < \kappa$, we identify a small subset of $\alpha^Z$ and show that it contains representative of every equivalence class modulo $U$. Let $D$ witness $\lambda$-density. Choose an antichain $A \subseteq D$ of size $< \! \lambda$, and choose $f : A \to \alpha$.  There are $\sum_{\gamma < \lambda} \lambda^\gamma \cdot \alpha^\gamma = 2^{<\lambda}$ many choices.  Using $\kappa$-completeness, let $\{ B_\beta : \beta < \alpha \}$ be pairwise disjoint and such that each $[B_\beta]_I = \sum f^{-1}(\beta)$.  Let $g_f : Z \to \alpha$ be defined by $g_f(z) = \beta$ if $z \in B_\beta$ and $g_f(z) = 0$ if $z \notin \bigcup_{\beta<\alpha} B_\beta$.

Now let $g : Z \to \alpha$ be arbitrary.  By $\kappa$-completeness, $A = \{ g^{-1}(\beta) :\beta < \alpha$ and $g^{-1}(\beta) \in I^+ \}$ forms a maximal antichain.  Let $A' \subseteq D$ be a maximal antichain refining $A$.  There is some $A'' \subseteq A'$ of size $< \! \lambda$ such that $\sum A'' \in U'$.  Let $f : A'' \to \alpha$ be defined by $f(a) = \beta$ iff $a \leq_I g^{-1}(\beta)$.  If $[B]_I = \sum A''$, then $\{ z \in B : g(z) \not= g_f(z) \} \in I$, so $g =_U g_f$.
\end{proof}

The following contrasts with the consistency results of Section 2:

\begin{corollary}Suppose $\mu$ is a singular cardinal such that $2^{\cf(\mu)} < \mu$, $\lambda$ is regular, and $2^{<\lambda} < 2^\lambda$.  Then there is no normal and fine, $\lambda$-dense ideal on $\p_{\mu^+}(\lambda)$.  Furthermore, there is a proper class of such $\lambda$.
\end{corollary}

\begin{proof}
Suppose such an ideal exists, and let $U \supseteq I^*$ be given by Lemma \ref{huberich}.  Then $| \prod \mu /U| \leq 2^{<\lambda}$, and $U$ is $(\cf(\mu)+1,\lambda)$-regular.  Theorem~\ref{keisler} implies that $| \prod 2^{\cf(\mu)} / U | \geq 2^\lambda$, a contradiction.

Assume for a contradiction that $\alpha$ is such that $2^{<\lambda} = 2^\lambda$ for all regular $\lambda \geq \alpha$.  Let $\kappa = 2^\alpha$.  We will show by induction the impossible conclusion that $2^\beta = \kappa$ for all $\beta \geq \alpha$.  Suppose that this holds for all $\gamma < \beta$. If $\beta$ is regular, $2^{<\beta} = 2^\beta$ by assumption, so $2^\beta = \kappa$.  If $\beta$ is singular, then by \cite[Theorem 5.16]{jechbook} $2^\beta = (2^{<\beta})^{\cf(\beta)} = \kappa^{\cf(\beta)}$.  If $\cf(\beta) < \gamma < \beta$, then $\kappa = 2^\gamma = (2^\gamma)^{\cf(\beta)} = \kappa^{\cf(\beta)}$. 
\end{proof}

\begin{corollary}
\label{incon1}
If $\kappa$ is singular such that $2^{\cf(\kappa)} < \kappa$, then there is no uniform, $\kappa^+$-complete, $\kappa^+$-dense ideal on $\kappa^{+n}$ for $n \geq 2$.
\end{corollary}

\begin{proof}Assume $I$ is a uniform, $\kappa^+$-complete, $\kappa^+$-dense ideal on $\kappa^{+n}$ for some $n \geq 2$.  Define $\phi : \p(\kappa^+) \to {\p(\kappa^{+n})/I}$ by $X \mapsto || \kappa^+ \in j(X) ||_{\p(\kappa^{+n})/I}$.  Let $J = \ker \phi$.  $\phi$ lifts to an embedding of $\p(\kappa^+)/J$ into $\p(\kappa^{+n})/I$.  Since $J$ is clearly normal and $\kappa^{++}$-saturated, the embedding is regular, since for a maximal antichain $\{ A_\alpha : \alpha < \kappa^+ \}$, $\Vdash \kappa^+ \in j(\nabla_{\alpha < \kappa^+} A_\alpha)$, so it is forced that for some $\alpha<\kappa^+$, $\phi(A_\alpha)$ is in the generic filter.  Thus $J$ is a normal $\kappa^+$-dense ideal on $\kappa^+$.  We have $2^{<\kappa^+} < 2^{\kappa^+}$ by Corollary~\ref{wg}, so $\kappa$ cannot be singular such that $2^{\cf(\kappa)} < \kappa$.   
\end{proof}

These methods can also be used to deduce more cardinal arithmetic consequences of dense ideals.  First we need a few more lemmas:

\begin{theorem}[Kunen-Prikry \cite{kp}]
\label{kp}
If $\kappa$ is regular and $U$ is a $(\kappa^+,\kappa^+)$-regular ultrafilter, then $U$ is $(\kappa,\kappa)$-regular.
\end{theorem}

\begin{lemma}
\label{linear}
Suppose $( L, < )$ is a linear order such that for all $x \in L$, \break $| \{ y \in L : y < x \} | \leq \kappa$.  Then $|L| \leq \kappa^+$.
\end{lemma}

\begin{corollary}
\label{wg2}
Suppose there is a $\kappa^+$-complete, $\kappa^+$-dense ideal on $\kappa^{+n}$, where $n \geq 2$.  Then for $0 \leq m \leq n$, $2^{\kappa^{+m}} = \kappa^{+m+1}$.
\end{corollary}

\begin{proof}
Let $I$ be such an ideal, and let $U \supseteq I^*$ be given by Lemma~\ref{huberich}.  By Lemma~\ref{ultsize}, $|\kappa^{\kappa^{+n}} / U | \leq 2^\kappa$, which is $\kappa^+$ by Corollary~\ref{wg}.  Note that for any cardinal $\mu$, any ultrafilter $V$ on a set $Z$, and any $g : Z \to \mu^+$,  $\{ [f]_V : f <_V g \}$ has cardinality at most $| \mu^Z / V |$.   Thus, applying Lemma~\ref{linear} inductively, we get that $| (\kappa^{+m})^{\kappa^{+n}} / U | \leq \kappa^{+m+1}$ for all $m < \omega$.

$U$ is $(\kappa^{+n},\kappa^{+n})$-regular, so by Theorem~\ref{kp}, it is $(\kappa^{+m},\kappa^{+m})$-regular for $m \leq n$.  Assume for induction that $2^{\kappa^{+r}} = \kappa^{+r+1}$ for $r < m \leq n$; note the base case $m = 1$ holds.  Let $\{ X_\alpha : \alpha < \kappa^{+m} \}$ witness $(\kappa^{+m},\kappa^{+m})$-regularity, and let $\beta_z = \ot(\{ \alpha : z \in X_\alpha \})$.  By Theorem~\ref{keisler} and the above observations, we have:
\[ 2^{\kappa^{+m}} \leq |\prod 2^{\beta_z} / U| \leq |\prod 2^{\kappa^{+m-1}} / U| = |\prod \kappa^{+m} / U| \leq \kappa^{+m+1} \mbox{. } \] \end{proof}

We note that if the hypothesis of Corollary~\ref{wg2} is consistent, then no cardinal arithmetic above $\kappa^{+n}$ can be deduced from it, since any forcing which adds no subsets of $\kappa^{+n}$ will preserve the relevant properties of the ideal.

By combining this technique with the results of Section 2, we can answer the following, which was Open Question 16 from \cite{foremanhandbook}:

\begin{question}[Foreman]
Is it consistent that there is a uniform ultrafilter $U$ on $\omega_3$ such that $\omega^{\omega_3} /U$ has cardinality $\omega_3$? Is it consistent that there is a uniform ultrafilter $U$ on $\aleph_{\omega+1}$ such that $\omega^{\aleph_{\omega+1}} /U$ has cardinality $\aleph_{\omega+1}$? Give a characterization of the possible cardinalities of ultrapowers.
\end{question}

\begin{theorem}Assume ZFC is consistent with a super-almost-huge cardinal.  Then it is consistent that every regular uncountable cardinal $\kappa$ carries a uniform ultrafilter $U$ such that $|\omega^\kappa / U| = \kappa$. 
\end{theorem}

This follows from Section 2 and the next result.

\begin{lemma}Suppose $\kappa = \mu^+$, GCH holds at cardinals $\geq \mu$, and for all regular $\lambda \geq \kappa$, there is a normal and fine, $\kappa$-complete, $\lambda$-dense ideal on $\p_\kappa(\lambda)$.  Then for every regular $\lambda$, there is a uniform ultrafilter $U$ on $\lambda$ such that $|\mu^\lambda / U| = \lambda$.
\end{lemma}

\begin{proof}
Let $I$ be a normal and fine, $\kappa$-complete, $\lambda$-dense ideal on $Z = \p_\kappa(\lambda)$, where $\kappa= \mu^+$ and $\lambda$ is regular.  Note that $|Z| = \lambda$, and every $Y \subseteq Z$ of size $<\lambda$ is in $I$.   Let $U \supseteq I^*$ be given by Lemma~\ref{ultsize}, so that $| \mu^Z / U | \leq 2^{<\lambda}$.  If $2^{<\lambda} = \lambda$, then $| \mu^Z / U | \leq \lambda$.  Since $2^\mu = \kappa$ and any ultrafilter extending $I^*$ is $(\kappa,\lambda)$-regular, Theorem~\ref{keisler} implies that $| \kappa^Z / U | > \lambda$, and Lemma~\ref{linear} implies that $| \kappa^Z / U | \leq | \mu^Z / U |^+$.  Thus $| \mu^Z / U | = \lambda$.   
\end{proof}

The following extra conclusion can be immediately deduced in the case of $\mu < \aleph_\omega$ and $\lambda = \rho^+$, where $\cf(\rho) = \omega$.  Suppose $\mu = \omega_n$.  Since $| \omega_{n+1}^Z / U | > \lambda$, we cannot have $| \omega_m^Z / U | < \rho$ for any $m$, since by Lemma~\ref{linear}, we would have $| \omega_r^Z / U | < \rho$ for all $r < \omega$.  Also, $U$ is $(\omega,\omega)$-regular, so Theorem~\ref{keisler} implies that $|\omega^Z /U| \geq |\omega^Z /U|^\omega$.  Thus $| \omega_m^Z / U | = \lambda$ for all $m \leq n$.

\section{Compatibility with square}
Solovay~\cite{solovay} showed that $\square_\delta$ fails when $\delta \geq \kappa$ and $\kappa$ is strongly compact.  In contrast, we will show $(\forall \delta \geq \kappa)\square_\delta$ is consistent with the kind of generically supercompact $\kappa$ constructed in Section 2.  The key difference is that nontrivial forcings may be absorbed into the quotient algebras of the ideals in the generic case.

We start with a model given by Section 2, force $\square$, and show that dense ideals still exist.  We will use the following variation on Foreman's duality theorem~\cite{foremanduality}.

\begin{lemma}
\label{dualabsorb}
Suppose $I$ is a precipitous ideal on $Z \subseteq \p(X)$ and $e : \mathbb{P} \to \mathcal{B}(\p(Z)/I)$ is a regular embedding.  Suppose that for all generic $G \subseteq \mathcal{B}(\p(Z)/I)$, if $j : V \to M \subseteq V[G]$ is the associated embedding and $H = e^{-1}[G]$, there is a filter $\hat{H} \in V[G]$ that is $j(\mathbb{P})$-generic over $M$ and such that $j[H] \subseteq \hat{H}$.  Then there is a $\mathbb{P}$-name for a precipitous ideal $J$ on $Z$ such that $\mathcal{B}( \mathbb{P} * \dot{\p(Z)/J}) \cong \mathcal{B}(\p(Z)/I)$.  Furthermore, $J$ has the same completeness and normality properties as $I$.
\end{lemma}

\begin{proof}
Let $H$ be $\mathbb{P}$-generic over $V$, and let $G$ be $\mathcal{B}(\mathcal{P}(Z)/I)$-generic over $V$ with $e[H] \subseteq G$.  Let $j : V \to M$ be the generic ultrapower embedding, and let $\hat{H}$ be as hypothesized.  Then $j$ is uniquely extended to $\hat{j} : V[H] \to M[\hat{H}]$.  In $V[H]$, let $\mathbb{Q} = \mathcal{B}(\mathcal{P}(Z)/I) / e[H]$, and let $J = \{ X \subseteq Z : 1 \Vdash_{\mathbb{Q}} [\id]_M \notin \hat{j}(X) \}$.  Let $\iota : \mathbb{P} * \dot{\p(Z)/J}$ be defined by $\iota(p,\dot{X}) = e(p) \wedge || [\id] \in \hat{j}(\dot{X}) ||$.

It is easy to see that $\iota$ is order and antichain preserving.  To see that the range of $\iota$ is dense, let $A \in I^+$ be arbitrary.  Take $G$ with $A \in G$, so that $[\id] \in j_G(A)$.  If $H = e^{-1}[G]$, then some $p \in H$ forces $A \in J^+$.  Let $\dot{X}$ be a $\mathbb{P}$-name such that $p \Vdash \dot{X} = \check{A}$ and $q \Vdash \dot{X} = Z$ whenever $q \perp p$; this makes sure $\dot{X}$ is forced to be in $J^+$.  Then $\iota(p,\dot{X}) \leq A$, since any $G$ with $e(p) \wedge || [\id] \in \hat{j}(\dot{X}) || \in G$ must have $[\id] \in j(A)$ and thus $A \in G$.

Suppose $H * \bar{G} \subseteq  \mathbb{P} * \p(Z) / \bar{I}$ is generic, and let $G * \hat{H} = \iota[H * \bar{G}]$.  For $A \in J^+$, $A \in \bar{G}$ iff $[\id]_M \in \hat{j}(A)$.  If $i : V[H] \to N = V[H]^Z / \bar{G}$ is the canonical ultrapower embedding, then there is an elementary embedding $k : N \to M[\hat{H}]$ given by $k([f]_N) = \hat{j}(f)([\id]_M)$, and $\hat{j} = k \circ i$.  Thus $N$ is well-founded, so $J$ is precipitous.  If $f : Z \to \ord$ is a function in $V$, then $k([f]_N) = j(f)([\id]_M) = [f]_M$.  Thus $k$ is surjective on ordinals, so it must be the identity, and $N = M[\hat{H}]$.  Since $i = \hat{j}$ and $\hat{j}$ extends $j$, $i$ and $j$ have the same critical point, so the completeness of $J$ is the same as that of $I$.  Finally, since $[\id]_N = [\id]_M$, $I$ is normal in $V$ iff $J$ is normal in $V[H]$, because $j \restriction X = \hat{j} \restriction X$, and normality is equivalent to $[\id] = j[X]$.
\end{proof}

For a cardinal $\delta$, let $\mathbb{S}_\delta$ be the collection of bounded approximations to a $\square_\delta$ sequence.  That is, a condition is a sequence $\langle C_\alpha : \alpha \in \eta \cap \mathrm{Lim} \rangle$ such that $\eta < \delta^+$ is a successor ordinal, each $C_\alpha$ is a club subset of $\alpha$ of order type $\leq \delta$, and whenever $\beta$ is a limit point of $C_\alpha$, $C_\alpha \cap \beta = C_\beta$.  For proof of the following lemma, we refer the reader to \cite{sssr}.

\begin{lemma}For every cardinal $\delta$, $\mathbb{S}_\delta$ is countably closed and $(\delta+1)$-strategically closed and adds a $\square_\delta$ sequence $\langle C_\alpha : \alpha \in \delta^+ \cap \mathrm{Lim} \rangle = \bigcup G$, where $G \subseteq \mathbb{S}_\delta$ is the generic filter.  For every regular $\lambda \leq \delta$, there is a $\mathbb{S}_\delta$-name for a ``threading'' partial order $\mathbb{T}_\delta^\lambda$ that adds a club $C \subseteq (\delta^+)^V$ of order type $\lambda$ and such that whenever $\alpha$ is a limit point of $C$, $C \cap \alpha = C_\alpha$.  Furthermore, $\mathbb{S}_\delta * \mathbb{T}_\delta^\lambda$ has a $\lambda$-closed dense subset of size $2^\delta$.
\end{lemma}

\begin{theorem}
Suppose $\kappa$ is super-almost-huge and $\mu < \kappa$ is regular.  Then there is a $\mu$-distributive forcing extension in which $\kappa = \mu^+$, $\square_\lambda$ holds for all cardinals $\lambda \geq \kappa$, and for all regular $\lambda \geq \kappa$ there is a normal, fine, $\kappa$-complete, $\lambda$-dense ideal on $\p_\kappa(\lambda)$.
\end{theorem}

\begin{proof}
By Section 2, we may pass to a $\mu$-distributive forcing extension in which $\kappa = \mu^+$ and for all regular $\lambda \geq \kappa$ there is a normal, fine, $\kappa$-complete, $\lambda$-dense ideal on $\p_\kappa(\lambda)$, and GCH holds above $\mu$.  Over this model, force with $\mathbb{P}$, the Easton support product of $\mathbb{S}_\lambda$ where $\lambda$ ranges over all cardinals $\geq \kappa$.  For every cardinal $\lambda$, $\mathbb{P}$ naturally factors into $\mathbb{P}_{<\lambda} \times \mathbb{P}_{\geq \lambda}$.   Note that if $\lambda \geq \kappa$, $\mathbb{P}_{\geq \lambda}$ is $(\lambda + 1)$-strategically closed.

First we show that for each regular $\lambda \geq \kappa$, $\mathbb{P}_{\geq \lambda}$ is $\lambda^+$-distributive in $V^{\mathbb{P}_{<\lambda}}$.  Suppose that $H_0 \times H_1$ is $(\mathbb{P}_{<\lambda} \times \mathbb{P}_{\geq \lambda})$-generic, and $f : \lambda \to \ord$ is in $V[H_0][H_1]$.  Then in $V[H_1]$, there is a $\mathbb{P}_{<\lambda}$-name $\tau$ for $f$.  By GCH and the fact that we take Easton support, $|\mathbb{P}_{<\lambda}| = \lambda$, so it is $\lambda^+$-c.c.\ in $V[H_1]$.  Thus $\tau$ may be assumed to be a subset of $V$ of size $\lambda$.  By the strategic closure of $\mathbb{P}_{\geq \lambda}$, $\tau \in V$.  Thus $f = \tau^{H_0} \in V[H_0]$, establishing the claim.

Next we show that $\mathbb{P}$ preserves all regular cardinals.  First note that since $\mathbb{P}$ is $(\kappa+1)$-strategically closed, $\mathbb{P}$ cannot change the cofinality of any regular $\delta$ to some $\lambda \leq \kappa$.  If $\mathbb{P}$ does not preserve regular cardinals, then in some generic extension $V[G]$, there are $\lambda < \delta$ which are regular in $V$ with $\kappa < \lambda$, such that $V[G] \models \cf(\delta) = \lambda$.  Let $H = H_0 \times H_1$, where $H_0 \subseteq \mathbb{P}_{<\lambda}$ and $H_1 \subseteq \mathbb{P}_{\geq \lambda}$.  By the $\lambda^+$-c.c.\ of $\mathbb{P}_{<\lambda}$, $V[H_0] \models \cf(\delta) > \lambda$, and by the $\lambda^+$-distributivity of $\mathbb{P}_{\geq \lambda}$ in $V[H_0]$, $V[H] \models \cf(\delta) > \lambda$, a contradiction.  Since a square sequence is upwards absolute to models with the same cardinals and $\mathbb{S}_\lambda$ regularly embeds into $\mathbb{P}$ for all $\lambda \geq \kappa$, $\mathbb{P}$ forces $(\forall \lambda \geq \kappa) \square_\lambda$.

For each regular $\lambda \geq \kappa$, let $Z_\lambda = \p_\kappa(\lambda)$.  We want to show that in $V^{\mathbb{P}}$, for each regular $\lambda \geq \kappa$, there is a normal, fine, $\lambda$-dense ideal on $Z_\lambda$.  It suffices to show that such an ideal exists in $V^{\mathbb{P}_{<\lambda}}$, since $\mathbb{P}_{\geq \lambda}$ adds no subsets of $\lambda$, and $|Z_\lambda | = \lambda$.  First note that by the strategic closure of $\mathbb{P}$, the dense ideal on $\kappa$ is unaffected.

Let $\mathbb{Q}$ be the Easton support product of $\mathbb{S}_\lambda * \mathbb{T}_\lambda^\mu$, where $\lambda$ ranges over all cardinals. There is a coordinate-wise regular embedding of $\mathbb{P}$ into $\mathbb{Q}$.  When $\lambda$ is regular, $\mathbb{Q}_{<\lambda}$ has a dense $\mu$-closed subset of size $\lambda$.  Hence it regularly embeds into $\mathcal{B}(\col(\mu,\lambda))$.  The dense ideal $I_\lambda$ on $Z_\lambda$ in $V$ has quotient algebra isomorphic to $\mathcal{B}(\mathbb{R} \times \col(\mu,\lambda))$ for some small $\mathbb{R}$, and so $\mathbb{Q}_{<\lambda}$ regularly embeds into this forcing.

If $G \subseteq \p(Z_\lambda)/I_\lambda$ is generic, let $H$ be the induced generic for $\mathbb{P}_{<\lambda}$, and let $j : V \to M \subseteq V[G]$ be the ultrapower embedding.  Recall that $\crit(j) = \kappa$, $j(\kappa) = \lambda^+$, $\lambda^{++}$ is a fixed point of $j$, and $j[\lambda] \in M$.  First note that $j[\lambda] \setminus j(\kappa)$ is an Easton set in $M$.  If $j(\kappa) \leq \delta \leq j(\lambda)$ and $\delta$ is regular in $M$, then since $\ot(j[\lambda] \cap \delta) \leq \lambda < \delta$, $\sup(j[\lambda] \cap \delta) < \delta$.

For each cardinal $\delta$ such that $\kappa \leq \delta < \lambda$, let $\langle C_\alpha^\delta : \alpha < \delta^+ \rangle$ be the $\square_\delta$ sequence and let $t_\delta$ be the ``thread'' of order type $\mu$, both given by $H \restriction ( \mathbb{S}_\delta * \mathbb{T}^\mu_\delta)$.  By the $\mu$-distributivity of $\mathbb{S}_\delta * \mathbb{T}_\delta^\mu$, all initial segments of $t_\delta$ are in $V$, and since they are small, $j(t_\delta \cap \alpha) = j[t_\delta \cap \alpha]$ for $\alpha < \delta^+$, and $j$ is continuous at all limit points of $t_\delta$.  Let $\gamma_\delta = \sup(j[\delta^+]) < j(\delta^+)$, and in $M$ consider $m_\delta = \bigcup_{\alpha < \delta^+} j(\langle C_\alpha^\delta : \beta < \alpha \rangle) \cup \{ (\gamma_\delta, j[t_\delta] ) \} $.  Each $m_\delta$ is a condition in $(\mathbb{S}_{j(\delta)})^M$, and the sequence $m = \langle m_\delta : \delta \in j[\lambda] \setminus \kappa \cap \card^M \rangle$ is a condition in $(\mathbb{P}_{<j(\lambda)})^M$.

$M$ thinks $\mathbb{P}_{<j(\lambda)}$ is $j(\kappa)$-strategically closed, and this is true in $V[G]$ as well since these models share the same $\lambda$-sequences, and $j(\kappa) = \lambda^+$ in $V[G]$.  Since $j(\lambda^+) < (\lambda^{++})^V$, $\p(\mathbb{P}_{<j(\lambda)})^M$ has cardinality $\lambda^+$ in $V[G]$.  Thus we may use the winning strategy to build a filter $\hat{H} \subseteq (\mathbb{P}_{<j(\lambda)})^M$ that is generic over $M$, with $m \in \hat{H}$.  Since $m$ is a lower bound to $j(p)$ for all $p \in H$, we have $j[H] \subseteq \hat{H}$.

Therefore, the hypotheses of Lemma~\ref{dualabsorb} are satisfied with respect to $I_\lambda$, $Z_\lambda$, and $\mathbb{P}_{<\lambda}$.  Thus we have a $\mathbb{P}_{<\lambda}$-name for a normal and fine ideal $J_\lambda$ on $Z_\lambda$ such that $\mathcal{B}(\mathbb{P}_{<\lambda} * \dot{\p(Z_\lambda)/J_\lambda)} \cong \mathcal{B}(\p(Z_\lambda)/I_\lambda)$.  Hence $J_\lambda$ is $\lambda$-dense in $V[H]$.
\end{proof}

We note that in the case $\kappa = \omega_1$, $\p(Z_\lambda)/J_\lambda \cong \mathcal{B}(\col(\omega,\lambda))$ for all regular $\lambda$.  But for higher cardinals, Proposition~\ref{Z-refl} shows the quotient algebras must differ from those given by Theorem~\ref{maindense}.  The crux is that the ``threading'' forcings are left over as regular suborders.

It is not possible to improve this result to get the consistency of, for example, ``For all cardinals $\lambda \geq \omega_1$, $\square_\lambda$ holds and there is a normal, fine, $\lambda$-dense ideal on $\p_{\omega_1}(\lambda)$.''  Burke and Matsubara~\cite{BM} showed that if $\cf(\lambda) < \kappa$ and there is normal, fine, $\kappa$-complete, $\lambda^+$-saturated ideal on $\p_\kappa(\lambda)$, then every stationary subset of $\lambda^+ \cap \cof({<}\kappa)$ reflects.

\bibliographystyle{amsplain.bst}
\bibliography{junebib}

\end{document}